\documentclass{article}

\usepackage{geometry}
\usepackage{hyperref}
\usepackage{amsmath,amssymb,amsthm}
\usepackage{mathtools}
\mathtoolsset{showonlyrefs}
\usepackage{xcolor}

\newtheorem{thm}{Theorem}[section]
\newtheorem{lem}[thm]{Lemma}
\newtheorem{prop}[thm]{Proposition}
\theoremstyle{remark}
\newtheorem{rem}[thm]{Remark}

\theoremstyle{definition}

\newtheoremstyle{Claim}{}{}{\itshape}{}{\itshape\bfseries}{:}{ }{#1}
\theoremstyle{Claim}

\newcommand{\Rset}{\mathbb{R}}

\newcommand{\saOaT}{\sqrt{\alpha_1 \alpha_2}}

\usepackage[latin1]{inputenc}
\theoremstyle{plain}
\newtheorem{proposition}[thm]{Proposition}
\newtheorem{theorem}[thm]{Theorem}
\newtheorem{corollary}[thm]{Corollary}
\newtheorem{lemma}[thm]{Lemma}

\theoremstyle{definition}
\newtheorem{definition}[thm]{Definition}
\theoremstyle{remark}
\newtheorem{remark}[thm]{Remark}
\newcommand{\eps}{\varepsilon}
\newcommand{\dis}{\displaystyle}

\newcommand{\N}{{\mathbb{N}}}

\newcommand{\R}{{\mathbb{R}}}

\newcommand{\from}{\colon}

\newcommand{\nut}{{\tilde\nu}}
\usepackage{bm}

\DeclareMathOperator{\Ker}{Ker}
\DeclareMathOperator{\Range}{R}

\DeclareMathOperator{\spann}{span}

\titlepage
\title{Bifurcation and segregation in quadratic two-populations Mean Field Games systems}
\author{Marco Cirant and Gianmaria Verzini}
\date{\today}

\begin{document}

\maketitle

\begin{abstract}
We search for non-constant normalized solutions to the semilinear elliptic system
\[
\begin{cases}
- \nu \Delta v_i + g_i(v_j^2) v_i = \lambda_i v_i,\quad v_i>0 & \text{in $\Omega$} \\
\partial_n v_i = 0 & \text{on $\partial \Omega$}\\
\int_\Omega v_i^2\,dx = 1, & 1\leq i,j\leq 2, \quad j\neq i,
\end{cases}
\]
where $\nu>0$, $\Omega \subset \Rset^N$ is smooth and bounded, the functions $g_i$ are positive and increasing, and both
the functions $v_i$ and the parameters $\lambda_i$ are unknown.
This system is obtained, via the Hopf-Cole transformation, from a two-populations ergodic Mean Field Games
system, which describes Nash equilibria in differential games with identical players. In these models, each population consists of a very large number of indistinguishable rational agents, aiming at minimizing some long-time average criterion.

Firstly, we discuss existence of nontrivial solutions, using variational methods when $g_i(s)=s$, and bifurcation ones in the general case;
secondly, for selected families of nontrivial solutions, we address the appearing of segregation in the vanishing viscosity limit, i.e.
\[
\int_{\Omega} v_1 v_2 \to 0 \qquad \text{as }\nu\to0.
\]
\end{abstract}

\noindent
{\footnotesize \textbf{AMS-Subject Classification}}. {\footnotesize 35J47, 49N70, 35B25, 35B32}

\section{Introduction}

We consider the following semilinear elliptic system
\begin{equation}\label{mainsys}
\begin{cases}
- \nu \Delta v_1 + g_1(v_2^2) v_1 = \lambda_1 v_1 & \\
- \nu \Delta v_2 + g_2(v_1^2) v_2 = \lambda_2 v_2 & \text{in $\Omega$} \\
\int_\Omega v_1^2\,dx = \int_\Omega v_2^2\,dx = 1,\quad v_1,v_2>0\\
\partial_n v_1 = \partial_n v_2 = 0 & \text{on $\partial \Omega$}.
\end{cases}
\end{equation}
Here $\Omega\subset\R^N$ is a smooth bounded domain, normalized in such a way that
\[
|\Omega| = 1,
\]
$\nu>0$, and both the functions $v_i$ and the parameters $\lambda_i$ are unknown.
The interaction functions $g_i \in C^2([0, \infty))$ satisfy
\begin{equation}\label{Gass}
\begin{split}
\bullet\quad&C^{-1}_g s \le g_i(s) \le C_g s \qquad \forall s \ge 0,\\
\bullet\quad&g_i\text{ is strictly increasing, } \quad g_i'(1) > 0,
\end{split}
\end{equation}
for some $C_g > 0$ ($i=1,2$).

The elliptic system \eqref{mainsys} arises in the context of Mean Field Games (briefly MFG) theory. MFG is a branch of Dynamic Games which has been proposed independently by Lasry, Lions \cite{jeux1,jeux2,LasryLions}
and Caines, Huang, Malham\'e \cite{HCM,HCM2} in the engineering community, with the aim of modeling and analyzing decision processes involving a very large number of indistinguishable rational agents. Here, we focus on MFG with two competing populations, where every individual of the $i$-th population ($i=1,2$) is represented by a typical agent, and whose state is driven by the controlled stochastic differential equation
\[
d X^i_s = -a^i_s ds + \sqrt{2 \nut} \, d B^i_s,
\]
where $B^i_s$ are independent Brownian motions. The agent chooses her own velocity $a^i_s$ in order to minimize a cost of long-time-average form
\[
\mathcal{J}^i(X^i_0, a^1, a^2) = \liminf_{T \rightarrow \infty} \frac{1}{T} \int_0^T \mathbb{E}  \left[\frac{|a^i_s|^2}{2} + g_i(\left(\hat{m}_{ j})_s\right)\right] ds,
\]
where $\hat{m}_{ j}$ denotes the empirical density of the players belonging to the other population (i.e. $j=3-i$). It has been shown (see in particular \cite{MR3127148}) that equilibria of the game (in the sense of Nash) are captured by the following system of non-linear elliptic equations
\begin{equation}\label{MFG}
\begin{cases}
- \nut \Delta u_i(x) + \frac{1}{2} {|\nabla u_i(x)|^2} + \lambda_i = g_i(m_{ j}(x))  \\
- \nut \Delta m_i(x) -{\rm div}(\nabla u_i(x) \, m_i(x)) = 0  & \text{in $\Omega$}  \\
\int_\Omega m_i dx = 1, \, m_i > 0, & i=1,2.
\end{cases}
\end{equation}
The unknowns $u_i, \lambda_i$ provide the value functions of typical players and the average costs respectively. On the other hand, the unknowns $m_i$ represent the stationary distributions of players of the $i$-th population implementing the optimal strategy, that is the long time behavior of agents playing in an optimal way. We suppose that the state $X^i_s$ is subject to reflection at $\partial \Omega$; this motivates the Neumann boundary conditions.

Note that the individual cost $\mathcal{J}^i$ is increasing with respect to $\hat{m}_{ j}$, as we are supposing that $g_i$ is increasing. In other words, every agent is lead to avoid regions of $\Omega$ where an high concentration of competitors is present. For this reason, our MFG model is expected to show phenomena of segregation between the two populations. In particular, segregation should arise distinctly in the vanishing viscosity regime, namely when the Brownian noise (whose intensity is controlled by $\nut$) becomes negligible with respect to interactions. We will explore this aspect in terms of qualitative properties of the two distributions $m_1, m_2$.

Another key feature of this model is the quadratic dependence of the cost $\mathcal{J}^i$ with respect to the velocity $\alpha_i$. It has been pointed out (see \cite{jeux1,LionsVideo}) that the so-called Hopf-Cole transformation partially decouples the equations in \eqref{MFG}, reducing the number of the unknowns. Precisely, if we let
\[
v_i^2 := m_i = e^{-u_i/\nut}\qquad \text{and}\qquad\nu = 2\nut^2
\]
then \eqref{MFG} becomes \eqref{mainsys}. We will therefore consider \eqref{mainsys} and transpose the obtained results to the original system \eqref{MFG}.

Before proceeding with the analysis of the reduced system \eqref{mainsys}, a few bibliographical remarks are in order. First of all, while the single population case has received a considerable attention, few papers deal with mathematical aspects of the multi-population setting. We mention that a preliminary study of \eqref{mainsys}-\eqref{MFG} has been made in \cite{MR3333058}, while a non-stationary version of \eqref{MFG} is considered in \cite{LachapelleWolfram}. The latter work provides also a motivation for \eqref{MFG} based on pedestrian crowd models. Our MFG system can be also seen as a simplified version of the population models presented in \cite{ABC}.

Since $|\Omega|=1$,
\[
v_1\equiv v_2\equiv1, \qquad \lambda_1=g_1(1), \quad \lambda_2=g_2(1)
\]
is a solution of \eqref{mainsys} for every value of $\nu$. We will refer to
it (or, with some abuse,
to the pair $(v_1,v_2)\equiv(1,1)$) as the \emph{trivial} (or
\emph{constant}) solution. The aim of our investigation is twofold:
firstly, to show the existence of families, indexed by $\nu$, of nontrivial
\emph{Nash equilibria} for \eqref{mainsys}; secondly, to analyze possible
\emph{segregation} phenomena for such families, as $\nu\to0$.
\begin{definition}\label{def:Nash}
The pair $(v_1,v_2)$ is a \emph{Nash equilibrium} for \eqref{mainsys}
if each $v_i$
achieves
\[
\lambda_i:=\inf\left\{\int_\Omega\left[\nu|\nabla w |^2 + g_i(v_j^2)w^2\right] \,dx : w\in
H^1(\Omega),\,\int_\Omega w^2\,dx =1\right\}.
\]
\end{definition}
It is easy to show (see Lemma \ref{lem:Nash_iff_sol} ahead) that a pair $(v_1,v_2)$ is a Nash equilibrium
if and only if (up to a change of sign of its components) it solves \eqref{mainsys} with multipliers $(\lambda_1,\lambda_2)$.
\begin{definition}\label{def:segr}
We say that a set of solutions
\[
\Sigma \subset
\left\{(\nu, v_1, v_2) \in \R \times C^{2,\alpha}(\overline{\Omega})\times C^{2,\alpha}(\overline{\Omega})  : \text{$(\nu, v_1,v_2)$
satisfies \eqref{mainsys} for some $(\lambda_1,\lambda_2)$} \right\}
\]
\emph{segregates} if it contains sequences $\{(\nu_n, v_{1,n}, v_{2,n})\}_{n}$ with $\nu_n \rightarrow 0$, and for
every such sequence it holds
\[
\int_\Omega v_{1,n} v_{2,n} \rightarrow 0 \quad \text{as $n \rightarrow \infty$}.
\]
\end{definition}
One important feature of system \eqref{mainsys} is that its unknowns are
both the functions $v_i$, which are required to be normalized (in the $L^2$
sense), and the parameters $\lambda_i$. Despite the large literature devoted
to existence results for semilinear elliptic systems, only few papers deal
with normalized solutions, mainly when searching for solitary waves
associated to nonlinear Schr\"odinger systems
\cite{MR2928850,MR3318740,MR3393268,Bartsch:2015ab,Bartsch:2015aa}.
Note that all these papers are based on variational methods, since the
systems they consider are of gradient type. This is not the case for
\eqref{mainsys}, except when the interactions $g_i$ are linear functions.

On the other hand, segregation issues have received much attention in the
last decade, and by now a large amount of literature
is dedicated to this subject, see e.g.
\cite{MR1939088,MR2090357,MR2146353,MR2278412,MR2393430,MR2384550,MR2599456,MR2831712,MR3062741,MR3116007,MR3375537}, the recent survey \cite{Soave:2015aa}, and references therein. Mainly
two types of competitions have been
widely investigated, namely the Lotka-Volterra type (e.g. $g_i(s)=a_i
\sqrt{s}$), and again the variational one. Furthermore in these papers
segregation (as defined in Definition \ref{def:segr}) is a first easy
step, while all the effort is done to show that the convergence of $v_1v_2$
to $0$ is very much stronger than merely $L^1$. Conversely, in our situation,
even the
$L^1$ convergence is not clear at all, mainly due to the unknown behavior
of the parameters $\lambda_i$. For instance, the set of trivial solutions
does not segregate at all. Actually, this is one of the main difficulties we
have to face.

Motivated by the above discussion, we first treat the variational case
\[
g_i(s)=\gamma_i s,\qquad\text{for some }\gamma_i>0.
\]
In such a case, as we mentioned, \eqref{mainsys} has a gradient structure, at least in dimension $N\leq3$:
Nash equilibria can be obtained as critical points of the functional
\[
I_\nu (v_1, v_2) = \int_\Omega\left[\frac{1}{\gamma_1}\,|\nabla v_1 |^2 +
\frac{1}{\gamma_2}\,|\nabla  v_2 |^2 + \frac{1}{\nu}\, v_1^2 v_2^2\right]\,dx
\]
constrained to the manifold
\[
M=\left\{(v_1,v_2)\in H^1(\Omega)\times H^1(\Omega) : \int_\Omega v_1 ^2\,dx = \int_\Omega v_1 ^2\,dx =1 \right\}.
\]
As a consequence, existence of solutions can be obtained by direct minimization of $I_\nu|_{{M}}$.
Regarding the asymptotic behaviour of such minimizers, using techniques contained in \cite{MR2928850} we can show $\Gamma$-convergence to the following limiting problem:
\begin{equation}\label{eq:Gamma_limit_pb}
\min\left\{\int_\Omega\left[\frac{1}{\gamma_1}\,|\nabla v_1 |^2 +
\frac{1}{\gamma_2}\,|\nabla  v_2 |^2 \right]\,dx: (v_1,v_2)\in {M},\ v_1\cdot v_2\equiv0,\right\}.
\end{equation}
It can be proved that such minimum is achieved, and, among other properties, that any minimizer $(V_1,V_2)$ is such that $V_1\sqrt{\gamma_2}-V_2\sqrt{\gamma_1}
\in C^{2,\alpha}(\overline{\Omega})$, for every $0<\alpha<1$ (see Proposition \ref{prop:limiting_problem} ahead).
As a matter of fact, we can prove the following.
\begin{theorem}[Variational case]\label{thm:intro_var}
Let $N\leq3$, $g_i(s)=\gamma_i s$,  $\gamma_i>0$, and let $\mu_1>0$ denote the first positive Neumann eigenvalue of $-\Delta$ in $\Omega$. Then, for every
\[
0 < \nu \leq \frac{\gamma_1\gamma_2}{\mu_1(\gamma_1+\gamma_2)},
\]
the minimum of $I_\nu|_{M}$ is achieved by a pair $(v_{1,\nu},v_{2,\nu})$, which is a nontrivial Nash equilibrium for \eqref{mainsys}.

Moreover, any family of minimizers exhibits segregation: up to subsequences,
\[
v_{i,\nu} \to V_{i} \  \text{in }H^1(\Omega) \cap
C^{\alpha}(\overline{\Omega})\qquad\text{as }\nu\to0,
\]
for every $\alpha<1$, where $(V_1,V_2)$ achieves \eqref{eq:Gamma_limit_pb}.
\end{theorem}

Turning to the general case, since \eqref{mainsys} has no variational structure, one is lead to search for solutions using topological methods.
In particular, it is  natural to use bifurcation theory to find nontrivial solutions $(\nu,v_1,v_2)$ branching off from the trivial
ones
\[
\mathcal{T} = \left\{(\nu,1,1):\nu>0\right\}\subset \R \times C^{2,\alpha}(\overline{\Omega})\times C^{2,\alpha}(\overline{\Omega}).
\]
We denote by $\mathcal{S}$ the closure of the set of nontrivial solutions of \eqref{mainsys}, so that a bifurcation point is a point of $\mathcal{S}\cap\mathcal{T}$. The classical bifurcation theory by Rabinowitz
\cite{MR0301587,MR0288640} can be applied
to our setting to obtain the following.
\begin{theorem}\label{thm:intro_bif}
Let $\mu^*>0$ denote a positive Neumann eigenvalue of $-\Delta$ in $\Omega$, and
\[
\nu^* = \frac{2\sqrt{g'_1(1)g'_2(1)}}{\mu^*}.
\]
\begin{itemize}
\item
If $\mu^*$ has odd multiplicity then there exists a continuum $\mathcal{C}^*\subset\mathcal{S}$ such that $(\nu^*,1,1)\in\mathcal{C}^*$ and
\begin{itemize}
\item either $(\nu^{**},1,1)\in\mathcal{C}^*$, where $\nu^{**}=2\sqrt{g'_1(1)g'_2(1)}/\mu^{**}$
and $\mu^{**}\neq \mu^*$ is another positive Neumann eigenvalue,
\item or $\mathcal{C}^*$ is unbounded; furthermore, in dimension $N\leq3$,
$\mathcal{C} \cap \{(\nu,v_1,v_2): \nu\geq \bar \nu\}$ is bounded for every $\bar\nu>0$,
and $\mathcal{C}$ contains a sequence
$(\nu_n,v_{1,n},v_{2,n})$ such that, as $n\to+\infty$,
\[
\nu_n\to0,\qquad \|(v_{1,n},v_{2,n})\|_{C^{2,\alpha}} \to +\infty.
\]
\end{itemize}
\item
If $\mu^*$ is simple (with eigenfunction $\psi^*$) then the set of non-trivial
solutions is, near $(\nu^*,1,1)$, a unique smooth curve with parametric representation
\[
\nu = \nu(\eps), \quad v = (1,1) + \eps v^* + o(\eps),
\]
where $\nu(0)= \nu^*$ and $v^* = \left(-\psi^*\sqrt{g'_1(1)} , \psi^*\sqrt{g'_2(1)} \right)$.
\end{itemize}
\end{theorem}
\begin{remark}
Sharper asymptotic expansions are provided in Remark \ref{rem:Taylor} ahead,
in case both $g_i$ are more regular.
\end{remark}
\begin{remark}\label{rem:notalwayssegregate}
To compare this theorem with the classical results by Rabinowitz, we recall that here the natural bifurcation parameter is $1/\nu$, rather than $\nu$ itself. In particular, in case infinitely many eigenvalues $\mu_n$ are odd,
we have infinitely many bifurcation points $\nu_n\to0$. As a consequence it
is easy to construct families of nontrivial solutions, jumping from branch to
branch, that not only do not
exhibit segregation, but even tend to the trivial solution as $\nu\to0$.
\end{remark}
The previous remark shows that one can not expect segregation for a generic
family of nontrivial solutions. It is then natural to ask whether segregation
occurs for the bifurcation branches above described, at least for the unbounded ones. According to Theorem \ref{thm:intro_bif}, in order to find unbounded branches of nontrivial solutions we first have to find odd eigenvalues of $-\Delta$ in $\Omega$,
and then to exclude that the corresponding branch goes back to the set of trivial solutions. Usually, in the bifurcation framework, both conditions can be
satisfied when working with the first eigenvalue of the linearized problem:
indeed, on one hand such eigenvalue is simple; on the other hand,
it is usually possible to carry over to the full branch the nodal characterization of the corresponding eigenfunction. Notice that this is not our case, since the first Neumann eigenvalue
is 0 and it does not provide a bifurcation point, while the first positive
eigenvalue $\mu_1$ is actually the second one. Another way to exploit these
ideas is to work in dimension $N=1$.
\begin{theorem}\label{thm:intro_1D}
Let $N=1$. For any $k\in\N$, $k\geq1$ there exists a continuum $\mathcal{C}_k$
of solutions, such that:
\begin{itemize}
 \item if $(\nu,v_1,v_2)\in \mathcal{C}_k$ then both $v_i$ have exactly $k-1$ critical
 points;
 \item $\overline{\mathcal{C}_k} \cap \mathcal{T} 
 = \left\{\left(
 \frac{2\sqrt{g'_1(1)g'_2(1)}}{\pi^2k^2},1,1\right)\right\}$;
 \item $h\neq k$ implies $\mathcal{C}_h \cap \mathcal{C}_k = \emptyset$;
 \item each $\mathcal{C}_k$ contains sequences with $\nu\to0$ and
 $\|(v_1,v_2)\|_{C^{2,\alpha}} \to +\infty$;
 \item each $\mathcal{C}_k$ segregates.
\end{itemize}
\end{theorem}
Once segregation is obtained, we have that the segregating branches converge,
up to subsequences,
to some limiting profiles. As a consequence, some natural questions arise,
about the type of convergence as well as about the properties of the limiting
profiles. We can give the full picture in the case of the first branch.
\begin{theorem}\label{thm:intro_S1}
Let $N=1$ and $\mathcal{C}_1$ as in Theorem \ref{thm:intro_1D}. Then, any sequence $\{(\nu_n, v_n, \lambda_n)\}_{n} \subset\mathcal{C}_1$ such that $\nu_n\to0$ is uniformly bounded in Lipschitz norm, and it holds
\[
v_{i,n} \to V_{i} \  \text{in }H^1(\Omega) \cap
C^{\alpha}(\overline{\Omega})\qquad\text{as } \nu_n \to0,
\]
for every $\alpha<1$, where $(V_1,V_2)$ is the minimizer (unique up to reflections) achieving \eqref{eq:Gamma_limit_pb} with $\gamma_i=g'_i(0)
\geq C_g^{-1}$.
\end{theorem}
\begin{remark}
We expect that most of the results of Theorems \ref{thm:intro_1D} and
\ref{thm:intro_S1} can be extended to higher dimension, in the radial
setting.
\end{remark}
It is easy to see that the convergence above is optimal: indeed, in case of
Lipschitz convergence, both $V_i$ would be $C^1$, a contradiction with
their explicit expression provided in Proposition
\ref{prop:limiting_problem}. Up to our knowledge, this is the first paper
obtaining optimal bounds for competitions which are not of power-type, even
though only in dimension $N=1$ (or in the radial case). The only other paper
dealing with generic competitions is \cite{MR2384550}, where uniform bounds
in the planar case $N=2$, not necessarily radial, are obtained.

Let us also point out that along the first branch the problem --which
is not variational-- inherits a variational principle in the limit.
This is a remarkable fact, since it shows a deep connection between the
variational problem \eqref{eq:Gamma_limit_pb} and the nonvariational system
\eqref{mainsys}.
This phenomenon was already observed, in a different situation, in
\cite{MR2283921}.

Of course, all the results we obtained for system \eqref{mainsys} can be
restated for the original MFG system \eqref{MFG}, recalling that
\[
m_i = v_i^2,\qquad u_i = - 2\nut \ln v_i.
\]

Finally, let us also mention that the true multidimensional case $N\geq2$,
as well as the case of $3$ or more populations, are of interest: they
will be the object of future studies.

The present paper is structured as follows: in Section \ref{sec:prelim} we
list a few preliminary results; Section \ref{sec:var} is devoted to the
analysis of the variational case, and to the proof of Theorem
\ref{thm:intro_var}, while Section \ref{sec:bifurc} contains the bifurcation
arguments and the proof of Theorem \ref{thm:intro_bif}; the Sturm-type
characterization of the nontrivial solutions in dimension $N=1$ is developed
in Section \ref{sec:1d}, and the proof of Theorem \ref{thm:intro_1D} is
completed in Section \ref{sec:1dsegr}, by showing segregation; finally, the
proof of Theorem \ref{thm:intro_S1} is contained in Section
\ref{sec:firstbranch}.

\textbf{Notation.}  Throughout the paper, $i$ denotes an index between $1$ and $2$, and $j=3-i$. With a little abuse of terminology,
we say that $(v_1,v_2)$ solves \eqref{mainsys} (or even that $(\nu, v_1, v_2)$ does) if there exist $\lambda_1,\lambda_2$ such that $(v_1,v_2,\lambda_1,\lambda_2)$ satisfies
\eqref{mainsys} (for some prescribed $\nu$).

We will denote by $(\mu_k)_{k \ge 0}$ the non decreasing sequence of the eigenvalues
of $-\Delta$ with homogeneous Neumann boundary conditions, namely $\mu_k$ is such that
\begin{equation}\label{eigen_sys}
\begin{cases}
- \Delta \psi_k = \mu_k \psi_k & \text{in $\Omega$}, \\
\partial_n \psi_k = 0 & \text{on $\partial \Omega$},
\end{cases}
\end{equation}
for some eigenvector $\psi_k \in C^{2, \alpha}(\overline{\Omega})$, which constitute
an orthonormal basis of $L^2(\Omega)$.
The first eigenvalue $\mu_0 = 0$ is simple and its corresponding
eigenfunction is $\psi \equiv 1$.

Given a function $u$, $u^\pm(x) = \max(\pm u(x),0)$ denote its positive and negative parts. Finally, $C, C_1, C_2,\dots$ denote (positive) constants we need not to specify.

\section{Preliminaries}\label{sec:prelim}

In this section we collect some preliminary results and some estimates of frequent use.
\begin{lemma}\label{lem:Nash_iff_sol}
The pair $(v_1,v_2)$ is a Nash equilibrium if and only if, up to a change of sign of each component, it is a (classical) solution of \eqref{mainsys}.
\end{lemma}
\begin{proof}
Considering $v_j$ as fixed, we have that $v_i$ is an $L^2$-normalized
eigenfunction of the Neumann realization of the operator
\[
H^1(\Omega) \ni w \mapsto -\nu \Delta w + g_i(v_j^2) w,
\]
and that $\lambda_i$ is the corresponding eigenvalue. But then $v_i$ is strictly
positive (up to a change of sign) if and only if it is the first eigenfunction, i.e.
it achieves the infimum in Definition \ref{def:Nash}. In particular, the proof of the
strict positivity in $\overline{\Omega}$ is a routine application of the Maximum Principle
and Hopf's Lemma.
\end{proof}

\begin{lemma}
Let $(v_1,v_2)$ solve \eqref{mainsys}. Then either it is trivial, or
\[
\min_{\overline{\Omega}} g_i(v_j^2)<\lambda_i<\max_{\overline{\Omega}} g_i(v_j^2),\qquad i=1,2.
\]
\end{lemma}
\begin{proof}
Integrating the equation for $v_i$ we can write
\[
\int_\Omega\left[\lambda_i-g_i(v_j^2)\right]v_i\,dx = \nu\int_{\partial\Omega}
\partial_n v_i\,d\sigma = 0,
\]
and since $v_i$ is positive, we deduce that either $\lambda_i-g_i(v_j^2)\equiv 0$,
i.e. $v_j$ is constant, or $\min_{\overline{\Omega}} g_i(v_j^2)<\lambda_i<
\max_{\overline{\Omega}} g_i(v_j^2)$.

Now, if both $v_i$ and $v_j$ are not constant, then the second alternative follows.
Let $v_i$ be constant: then its equation implies $g_i(v_j^2)\equiv \lambda_i$, so
that also $v_j$ is constant. Finally, both such constants must be $1$ by the $L^2$-
constraint (recall that $|\Omega|=1$).
\end{proof}
\begin{remark}
The above lemma shows that, for Nash equilibria, having a constant component
implies being the trivial solution (in this sense, the terminology
``constant solution'' is not ambiguous). In fact, if unique continuation for
\eqref{mainsys} holds, then any solution such that one component is
constant in a (non empty) open $\Omega_0\subset\Omega$ must be the trivial one. This
is always true, in particular, in dimension $N=1$ (see Section \ref{sec:1d}).
\end{remark}

\begin{lem}\label{lem:prel_estimates}
Let $(v_1, v_2)$ solve \eqref{mainsys}. The following identities hold, for every $i$:
\[
\begin{split}
&\nu \int_\Omega |\nabla  v_i|^2 +  \int_\Omega g_i(v_j^2) v_i^2 = \lambda_i;\\
&\nu \int_\Omega  \left|\dfrac{\nabla v_i}{v_i} \right|^2 + \lambda_i  =  \int_\Omega g_i(v_j^2).
\end{split}
\]
In particular, the multipliers $\lambda_i$ satisfy
\begin{equation}\label{identity1}
C_g^{-1} \int_\Omega v_1^2 v_2^2 \leq \lambda_i \leq C_g.
\end{equation}
\end{lem}
\begin{proof} To obtain the two identities it suffices to use integration by parts after multiplying the equation for $v_i$ by $v_i$ and $1/v_i$, respectively.
Since $\int_\Omega v_i^2 = 1$ and $C_g^{-1} s \leq g_i(s) \leq C_g s$, \eqref{identity1} follows.
\end{proof}
\begin{corollary}\label{coro:suff_cond_segr}
A sufficient condition for $\{(\nu_n,v_{1,n},v_{2,n})\}_n$ to segregate is that, for the corresponding
multipliers,
\[
\text{either }\lambda_{1,n} \rightarrow 0, \quad \text{or } \lambda_{2,n} \rightarrow 0,
\]
as $n \rightarrow \infty$.
\end{corollary}

\section{The variational case}\label{sec:var}

This section is devoted to the proof of Theorem \ref{thm:intro_var}. Such proof relies on ideas contained in \cite{MR2928850},
even though in that paper a different problem is considered (Dirichlet conditions, symmetric interaction, auto-catalytic reaction terms).
For this reason we describe the main ideas here, and refer the reader to \cite{MR2928850} for more details.

In the following we assume that $N\leq 3$ and
\[
g_i(s)=\gamma_i s, \qquad\gamma_i>0.
\]
As we already noticed, the corresponding system has a gradient structure. For easier notation we
make a change of variable, setting
\begin{equation}\label{eq:var_change_of_var}
\beta=\dfrac{1}{\nu},\quad
\tilde v_1 = \sqrt{\gamma_2} v_1,\quad
\tilde v_2 = \sqrt{\gamma_1} v_2.
\end{equation}
With this notation system \eqref{mainsys} becomes
\begin{equation}\label{eq:changofvar_varsys}
\begin{cases}
- \Delta \tilde v_1 + \beta\tilde v_2^2\tilde v_1 = \lambda_1 \tilde v_1\\
- \Delta \tilde v_2 + \beta\tilde v_1^2\tilde v_2 = \lambda_2 \tilde v_2 & \text{in $\Omega$}\\
\int_\Omega\tilde v_1^2 = \gamma_2, \  \int_\Omega\tilde v_2^2 = \gamma_1,\quad \tilde v_1,\tilde v_2>0\\
\partial_n\tilde v_1 = \partial_n\tilde v_2 = 0 & \text{on $\partial \Omega$}
\end{cases}
\end{equation}
(of course, the multipliers $\lambda_i$ here are suitable multiple of those of the original system).
Also for \eqref{eq:changofvar_varsys} positive solutions are Nash equilibria, among which the trivial one is the pair $(\sqrt{\gamma_2} ,\sqrt{\gamma_1})$.
Solutions to \eqref{eq:changofvar_varsys} are critical points of the functional
\[
J_\beta(\tilde v_1, \tilde v_2) = \int_\Omega\left[|\nabla \tilde v_1 |^2 +
|\nabla \tilde v_2 |^2 + \beta\tilde v_1^2\tilde v_2^2\right]
\]
constrained to the manifold
\[
\tilde M = \left\{ (\tilde v_1,\tilde v_2)\in H^1(\Omega)\times H^1(\Omega) : \int_{\Omega} \tilde v_1^2 = \gamma_2,\ \int_{\Omega} \tilde v_2^2 = \gamma_1 \right\}
\]
(recall that, since $N\leq 3$, the exponent $p=4$ is Sobolev subcritical and thus $J_\beta$ is of class $C^1$).
\begin{lemma}\label{lem:var_achieved}
For every $\beta>0$ the value
\[
 c_\beta:= \inf_{\tilde M} J_\beta\quad\text{ is achieved by } \quad
 (\tilde v_{1,\beta} , \tilde v_{2,\beta})\in \tilde M,
\]
which is a Nash equilibrium for \eqref{eq:changofvar_varsys}. Furthermore, if
\[
\beta\geq\frac{\gamma_1+\gamma_2}{\gamma_1\gamma_2}\mu_1
\]
(the first positive Neumann eigenvalue of $-\Delta$ in $\Omega$) then $(\tilde v_{1,\beta} , \tilde v_{2,\beta})$ is nontrivial.
\end{lemma}
\begin{proof}
Since $J_\beta$ is weakly l.s.c. in $H^1$, and $\tilde M$ is weakly closed,
the minima $(\tilde v_{1,\beta} , \tilde v_{2,\beta})$ exist by the direct
method. Moreover, since
\[
\int_\Omega\left[|\nabla \tilde v_i |^2 + \beta\tilde v_j^2\tilde v_i^2\right]
=
J_\beta(\tilde v_1, \tilde v_2) - \int_\Omega|\nabla \tilde v_j |^2,
\]
we have that such minima correspond to Nash equilibria for the original problem (the converse, of course, is false). We
are left to prove that, for $\beta$ large, $(\tilde v_{1,\beta} , \tilde v_{2,\beta})\neq(\sqrt{\gamma_2} ,\sqrt{\gamma_1})$.
To do that, we will choose a suitable competitor in the definition of $c_\beta$: let $\psi_1$ be an eigenfunction associated to $\mu_1$.
Then $\psi_1$ changes sign (indeed it is orthogonal to the eigenfunction $\psi_0=1$, associated to $\mu_0=0$) and we can find non-zero constants $a_\pm$ such that $(a_+\psi^+,a_-\psi^-)\in\tilde M$. Then
\[
c_\beta < J_\beta (a_+\psi^+,a_-\psi^-) = (\gamma_1+\gamma_2)\mu_1
\]
(equality can not hold since $(a_+\psi^+,a_-\psi^-)$ can not solve \eqref{eq:changofvar_varsys}) while
\[
J_\beta(\sqrt{\gamma_2} ,\sqrt{\gamma_1}) = \gamma_1 \gamma_2\beta.
\qedhere
\]
\end{proof}
Once we have solved the problem for $\beta>0$ fixed, we are ready to show $\Gamma$-convergence as $\beta\to+\infty$. Let
\[
J_\infty(\tilde v_1, \tilde v_2)  :=
\begin{cases}
\dis\int_\Omega\left[|\nabla \tilde v_1 |^2 + |\nabla \tilde v_2 |^2\right] &
\text{when }\dis\int_{\Omega} \tilde v_1^2\tilde v_2^2 = 0\smallskip\\
+\infty & \text{otherwise}
\end{cases}
\qquad\text{ and }\qquad
c_\infty := \inf_{\tilde M} J_\infty.
\]
\begin{lemma}\label{lem:var_conv}
As $\beta \to + \infty$,
\[
c_\beta \to c_\infty
\qquad\text{and (up to subs.) }\qquad
\tilde v_{i,\beta} \to \tilde V_{i} \text{ in }H^1(\Omega) \cap
C^{0,\alpha}(\overline{\Omega}),
\]
where $(\tilde V_{1} , \tilde V_{2})\in \tilde M$ achieves $c_\infty$.
\end{lemma}
\begin{proof}
First of all, we notice that, for every $(\tilde v_1,\tilde v_2)$ fixed,
\[
\beta_1 \leq \beta_2 \leq +\infty
\qquad\implies\qquad
J_{\beta_1}(\tilde v_1,\tilde v_2)\leq J_{\beta_2}(\tilde v_1,\tilde v_2).
\]
We deduce that $c_\beta$ is increasing in $\beta$ and bounded by $c_\infty$, thus it converges.
If the pair $(\tilde v_{1,\beta}, \tilde v_{2,\beta})$ achieves $c_\beta$, $\beta<+\infty$,
then $c_\beta\leq c_\infty$ implies
\[
\text{both }\| (\tilde v_{1,\beta}, \tilde v_{2,\beta}) \|_{H^1}^2 \leq c_\infty + \gamma_1+\gamma_2, \qquad
\text{and }\int_\Omega \tilde v_{1,\beta}^2 \tilde v_{2,\beta}^2 \leq \frac{c_\infty}{\beta}.
\]
We infer the existence of $(V_1,V_2)$ such that, up to subsequences, $\tilde v_{i,\beta} \to \tilde V_i$,
weakly in $H^1$ and strongly in $L^p$, $p=2,4$. In particular $(\tilde V_1,\tilde V_2)\in M$ and
$\tilde V_1\cdot \tilde V_2\equiv0$.
We have
\[
c_\infty \geq \lim c_\beta = \lim J_\beta (\tilde v_{1,\beta}, \tilde v_{2,\beta}) \geq
\liminf \int_{\Omega} |\nabla \tilde v_{1,\beta}|^2 + |\nabla \tilde v_{2,\beta}|^2
\geq
\int_{\Omega} |\tilde V_1|^2 + |\tilde V_2|^2 \geq c_\infty.
\]
Thus $(\tilde V_1, \tilde V_2)$ achieves $c_\infty$, and the inequalities above are indeed equalities, proving
convergence in $H^1$ norm and hence strong $H^1$ convergence.

The last thing to prove is the boundedness in $C^{0,\alpha}$
(which will imply convergence in $C^{0,\alpha}$ too, by Ascoli's Theorem). Notice that $(\tilde v_{1,\beta},\tilde v_{2,\beta})$ satisfies \eqref{eq:changofvar_varsys},
and that $0\leq \lambda_i \leq c_\infty/\gamma_i$. As a consequence, boundedness of the H\"older seminorm can be obtained as in \cite[Theorem 1.1]{MR2599456}, which provides
the same result in the case of Dirichlet boundary conditions: since the proofs in \cite{MR2599456} use blow-up arguments, in order to cover the Neumann case one just has to replace
odd extensions (from the half-space to $\R^N$) with even ones. More precisely, this replacement has to be performed in \cite[Lemmas 3.4, 3.5, 3.6]{MR2599456}.
\end{proof}
\begin{proof}[End of the proof of Theorem \ref{thm:intro_var}]
The proof of such theorem easily descends from Lemmas \ref{lem:var_achieved},
\ref{lem:var_conv}, when
going back to the original unknowns \eqref{eq:var_change_of_var}. In particular, notice that
$(\tilde V_1,\tilde V_2)\in\tilde M$ achieves $c_\infty$ if and only if $(\tilde V_1/
\sqrt{\gamma_2},\tilde V_2/\sqrt{\gamma_1})\in M$ achieves \eqref{eq:Gamma_limit_pb}.
\end{proof}
%
%
%
To conclude this section, we collect some properties of the minimizers associated to $c_\infty$.
\begin{proposition}\label{prop:limiting_problem}
Let $(\tilde V_1,\tilde V_2)\in\tilde M$ achieve $c_\infty$. Then $\tilde V_{1} \cdot \tilde V_{2} \equiv 0$ and there exist parameters $\Lambda_{i}$ such that
\[
-\Delta (\tilde V_1 - \tilde V_2 ) =
\Lambda_{1}\tilde V_1 - \Lambda_{2} \tilde V_2
\]
(in particular, $\tilde V_1 - \tilde V_2\in C^{2,\alpha}(\overline{\Omega})$).

Furthermore, in dimension $N=1$, let
$\Omega=(0,1)$. Then, the unique minimizer is (up to the reflection $x 	\leftrightarrow 1-x$)\begin{align*}
\tilde V_1(x) & = \sqrt{\frac{2\gamma_2}{x_0}}   \cos\left(\frac{\pi}{2x_0}x\right) \cdot \chi_{\left[0, x_0\right]}(x), \\
\tilde V_2(x) & = \sqrt{\frac{2\gamma_1}{1-x_0}} \cos\left(\frac{\pi}{2(1-x_0)}(1-x)\right) \cdot \chi_{\left[x_0,1\right]}(x),
\end{align*}
and $x_0 = \dfrac{\sqrt[3]{\gamma_2}}{
\sqrt[3]{\gamma_1}+\sqrt[3]{\gamma_2}}$.
\end{proposition}
\begin{proof}
Let
\[
J^*(w) = \int_{\Omega} |\nabla w|^2, \qquad
M^* = \left\{ w\in H^1(\Omega) : \int_{\Omega}
(w^+)^2 = \gamma_2,\ \int_{\Omega}
(w^-)^2 = \gamma_1 \right\}.
\]
Then, for component-wise positive pairs, $J_\infty(\tilde v_1,\tilde v_2)|_{\tilde M}
= J^*(\tilde v_1-\tilde v_2)|_{M^*}$, and the
first part of the proposition follows by the Lagrange multipliers rule (and by standard elliptic regularity).

Turning to the monodimensional case, we have that $(\tilde V_1,\tilde V_2)\in H^1(0,1)\times H^1(0,1)$ satisfies
\begin{equation}\label{eq:contoesplicitoindim1}
- (\tilde V_1 - \tilde V_2 )'' =
\Lambda_{1}\tilde V_1 - \Lambda_{2} \tilde V_2,\qquad \tilde V_1 \cdot \tilde V_2\equiv 0, \qquad \text{in }(0,1)
\end{equation}
with Neumann boundary conditions. By elementary considerations we deduce the existence of (at most countable) disjoint open intervals $I_{i,n}$,
with $i=1,2$ and $n\in \mathcal{N}_i\subset \N$, such that
\[
V_i(x) = \sum_{n\in \mathcal{N}_i} a_{i,n} \cos\left(\sqrt{\Lambda_i}( x - x_{i,n} )\right) \cdot \chi_{I_{i,n}}(x),
\]
where
\[
I_{i,n} =
\begin{cases}
\left(0,\frac{\pi}{2\sqrt{\Lambda_i}}\right) & \text{if } x_{i,n} =0\\
\left(x_{i,n} - \frac{\pi}{2\sqrt{\Lambda_i}},x_{i,n} - \frac{\pi}{2\sqrt{\Lambda_i}}\right) & \text{if } x_{i,n} \in \left(\frac{\pi}{2\sqrt{\Lambda_i}},1 - \frac{\pi}{2\sqrt{\Lambda_i}}\right)\\
\left(1 - \frac{\pi}{2\sqrt{\Lambda_i}},1\right) & \text{if } x_{i,n} =1
\end{cases}
,\qquad  \sum_n \frac{\pi}{2}|I_{i,n}|a_{i,n}^2 = \gamma_j.
\]
Now, also the pair defined by
\[
\tilde W_i = \frac{2\gamma_j}{\pi|I_{i,1}|a_{i,1}^2} \tilde V_i |_{I_{i,1}}
\]
achieves $c_\infty$; as a consequence, $\tilde W_1 - \tilde W_2$ solves \eqref{eq:contoesplicitoindim1}, while $\tilde W_1 - \tilde W_2\equiv0$
outside $I_{1,1}\cup I_{2,1}$. We deduce that both $\mathcal{N}_i$ are singleton, and finally that
\[
c_\infty = \min\left\{\int_0^1 (w_1')^2 + (w_2')^2 :
\begin{array}{l}
 w_1(x) =\sqrt{\frac{2\gamma_2}{x_1}}   \cos\left(\frac{\pi}{2x_1}x\right) \cdot \chi_{\left[0, x_1\right]}(x)\smallskip\\
 w_2(x) =\sqrt{\frac{2\gamma_1}{1-x_2}} \cos\left(\frac{\pi}{2(1-x_2)}(1-x)\right) \cdot \chi_{\left[x_2,1\right]}(x)\smallskip\\
 0<x_1\leq x_2<1
\end{array}
\right\},
\]
whose unique solution can be computed by elementary tools.
\end{proof}

\section{Bifurcation results}\label{sec:bifurc}

In this section we apply tools from global bifurcation theory in order to prove Theorem \ref{thm:intro_bif}. The main
references are the celebrated papers by Rabinowitz \cite{MR0301587} and Crandall and Rabinowitz \cite{MR0288640}, which
deal respectively with global bifurcation results for odd eigenvalues, and local ones for simple eigenvalues; for some
details about the asymptotic expansions in the latter case, we refer also to \cite[Chap. 5]{MR1225101}.
For the reader's convenience, we recall here the two statements we will apply.
\begin{theorem}[{\cite[Thm. 1.3]{MR0301587}}]\label{thm:Rabinowitz_odd}
Let $E$ be a Banach space, and let $G\from \R\times E \to E$, continuous and compact, be such that
\[
G(\beta,v) = \beta Lv + H(\beta,v),
\]
with $L$ linear and compact and $H(\beta,v)=o(\|v\|)$ as $v\to0$, uniformly on bounded $\beta$ intervals.

If $\beta^*$ is a characteristic value (i.e. $1/\beta^*$ is an eigenvalue) of $L$, having odd multiplicity, then
\[
\mathcal{S} := \overline{\left\{(\beta,v): v = G(\beta,v),\ v\neq0 \right\}}
\]
possesses a maximal subcontinuum $\mathcal{C}$ such that $(\beta^*,0)\in\mathcal{C}$, and $\mathcal{C}$ either
is unbounded in $E$, or $(\beta^{**}, 0)\in\mathcal{C}$, where $\beta^{**}\neq\beta^*$ is another characteristic value of $L$.
\end{theorem}
\begin{theorem}[{\cite[Thms. 1.37, 1.18]{MR0288640}}]\label{thm:Rabinowitz_simple}
Under the assumptions of Theorem \ref{thm:Rabinowitz_odd}, assume furthermore that $G$ is of
class $C^2$ and that $\partial^2_{\beta,v} G(\beta,0)=L$.

If $\beta^*$ is a simple characteristic value of $L$ and $v^*\neq0$ is such that
\[
\Ker(I-\beta^* L) = \spann\{v^*\},\qquad v^* \not\in \Range(I-\beta^* L),
\]
then $\mathcal{S}$ is a continuous curve, locally near $(\beta^*,0)$, parameterized as
\[
\eps\mapsto (\beta,v) = (\beta^* + \varphi(\eps), \eps v^* + \eps\psi(\eps)),
\]
where $\varphi(0)=0$, $\psi(0)=0$. If $G$ is more regular, then also the above curve is, and one can write higher order expansions (see Remark \ref{rem:Taylor}).
\end{theorem}
Among different possible choices, we will apply the above results in the ambient space
\[
E:=\left\{v=(v_1,v_2)\in C^{2,\alpha}(\overline{\Omega},\R^2) :\partial_{n}v_i = 0\text{ on }
\partial\Omega\right\}.
\]
\begin{lemma}\label{lem:smooth_G}
The map $G\from \R \times E \to E$ defined as
\[
u= G(\beta,v)
\qquad\iff\qquad
\begin{cases}
- \Delta u_1 + \beta g_1(v_2^2) u_1 = \lambda_1 u_1 & \\
- \Delta u_2 + \beta g_2(v_1^2) u_2 = \lambda_2 u_2 & \text{in $\Omega$} \\
\int_\Omega u_1^2\,dx = \int_\Omega u_2^2\,dx = 1,\quad u_1,u_2>0\\
\partial_n u_1 = \partial_n u_2 = 0 & \text{on $\partial \Omega$},
\end{cases}
\]
for suitable $\lambda_i$, is (well-defined and) of class $C^2$. Moreover it holds
\[
\partial_{v} G(\beta,1,1) = \beta L,
\]
where
\begin{equation}\label{eq:linearized_for_Rabinowitz}
z = L w
\qquad\iff\qquad
\begin{cases}
- \Delta z_1 = -2 g_1'(1) \left[w_2 -\int_\Omega w_2\right] & \\
- \Delta z_2 = -2 g_2'(1) \left[w_1 -\int_\Omega w_1\right]  & \text{in $\Omega$} \\
\int_\Omega z_1\,dx = \int_\Omega z_2\,dx = 0\\
\partial_n z_1 = \partial_n z_2 = 0 & \text{on $\partial \Omega$}.
\end{cases}
\end{equation}
\end{lemma}
\begin{proof}
The proof is based on standard smooth dependence of simple eigenvalues (and corresponding
eigenfunctions) with respect to the potentials, see for instance the book
\cite{MR2251558}. Let us consider
the map $F\from \R\times E\times E \times \R^2\to C^{0,\alpha}(\overline{\Omega},\R^2) \times
\R^2$,
\begin{equation}\label{eq:Fdefn}
F(\beta,v,u,\lambda) :=\left(
\begin{array}{c}
- \Delta u_1 + \beta g_1(v_2^2) u_1 - \lambda_1 u_1  \\
- \Delta u_2 + \beta g_2(v_1^2) u_2 - \lambda_2 u_2  \\
\int_\Omega u_1^2\,dx - 1\\
\int_\Omega u_2^2\,dx - 1
\end{array}
\right).
\end{equation}
Let $\beta,v$ be fixed. Then we can uniquely find positive eigenfunctions $u_i=u_i(\beta,v_j)$
and simple eigenvalues $\lambda_i=\lambda_i(\beta,v_j)$, such that $F=0$.
As a consequence, it is possible to apply the Implicit Function Theorem
in order to show that
\[
F(\beta,v,u,\lambda)=0
\qquad\iff\qquad
(u,\lambda)=\tilde G(\beta,v),
\]
with $\tilde G\in C^2$ (recall that each $g_i$ is of class $C^2$). More precisely, the
invertibility of $\partial_{(u,\lambda)} F$ at any of the points above mentioned can be obtained
by its injectivity (by Fredholm's Alternative).

Since $G$ is the projection of $\tilde G$ on $E$, the first part of the lemma follows. Observing
that
\[
\tilde G(\beta,1,1) = (1,1,\beta g(1), \beta g(1)),
\]
also the second part can be proved, by direct calculations.
\end{proof}
In order to apply the abstract results, we need to find the eigenvalues of the operator $L$
defined in the previous lemma. In the following, for easier notation, we write
\begin{equation}\label{abc}
\alpha_i = g'_i(1)>0 \qquad \text{(by assumption (\ref{Gass}))}.
\end{equation}
\begin{lemma}\label{lem:linear_G}
Let $L$ be defined as in \eqref{eq:linearized_for_Rabinowitz}. Then
\[
\beta^* L v^* = v^*,\quad v^*\neq0,
\qquad\iff\qquad
\beta^*=\frac{\mu^*}{2\sqrt{\alpha_1\alpha_2}},\quad v^* =
\left(-\sqrt{\alpha_1}\psi^*, \sqrt{\alpha_2}\psi^*\right),
\]
where $\mu^*$ is a positive Neumann eigenvalue of $-\Delta$ in $\Omega$ and $\psi^*$ a
corresponding eigenfunction.
\end{lemma}
\begin{proof}
Recall that $\beta^* L v^* = v^*$ if and only
if, for both $i$,
\[
\begin{cases}
-\Delta v^*_i = - 2\beta^* \alpha_i v^*_j & \text{in }\Omega\\
\int_\Omega v_i^* = 0\\
\partial_n v_i^*=0 & \text{on }\partial\Omega.
\end{cases}
\]
Setting
\[
\psi_\pm = \sqrt{\alpha_2} v_1 \pm \sqrt{\alpha_1}v_2,
\]
we obtain that the above system is equivalent to
\[
\begin{cases}
-\Delta \psi_\pm = \mp 2\beta^* \saOaT \psi_\pm & \text{in }\Omega\\
\int_\Omega \psi_\pm  = 0\\
\partial_n \psi_\pm =0 & \text{on }\partial\Omega.
\end{cases}
\]
Hence, if $\beta^*\neq0$, we infer that $\beta^*$ is a characteristic value of $L$ if
and only if $\psi_+ = 0$ (by the Maximum Principle) and $2\beta^* \saOaT$ is an eigenvalue
of $- \Delta$ with zero Neumann boundary conditions. Moreover, the characteristic vector
space associated to $\beta^*$ is generated by
\begin{equation}\label{eigenv}
\left(-\sqrt{\alpha_1}\psi^*, \sqrt{\alpha_2}\psi^*\right).
\end{equation}
Finally, note that $\beta^* = 0$ is not a characteristic value of $L$,
as $- \Delta \psi = 0$ and $\int_\Omega \psi = 0$ imply that $\psi \equiv 0$.
\end{proof}
The last ingredient we need is some control on the behavior of nontrivial solutions.
\begin{lemma} \label{lem:global_bound}
There exists a constant $C>0$ such that
\begin{enumerate}
 \item $\mathcal{S}\subset \left\{(\beta,v) : \int_\Omega|\nabla v|^2 \leq C \beta\right\}$;
 \item $\mathcal{S}\subset \left\{(\beta,v) : \beta \geq C \right\}$.
\end{enumerate}
\end{lemma}
\begin{proof}
Recalling Lemma \ref{lem:prel_estimates} we have that, in the present setting,
\[
\int_\Omega |\nabla  v_i|^2 \leq \lambda_i \leq  \beta\int_\Omega g_i(v_j^2) \leq C_g\beta,
\]
and the first inclusion follows. Concerning the second one, let by contradiction
$(\beta_n,v_n)_n\subset \mathcal{S}$ be such that $\beta_n\to0$. Then, by the first inclusion,
$v_n\to(1,1)$ in $H^1$ and, by elliptic regularity, also in $E$. We deduce that $\beta^*=0$
corresponds to a bifurcation point, and therefore $\partial_{v} (\cdot-G(\beta,\cdot)) =
I - \beta L$ can not be invertible at $\beta=0$, a contradiction.
\end{proof}
We are ready to prove our main bifurcation results.
\begin{proof}[Proof of Theorem \ref{thm:intro_bif}]
First of all, let $\mu^*$ be a positive Neumann eigenvalue, with odd multiplicity, and
\[
\beta^*=\frac{\mu^*}{2\sqrt{\alpha_1\alpha_2}}
\]
By Lemma \ref{lem:linear_G} we are in a position to apply Theorem
\ref{thm:Rabinowitz_odd}, obtaining a nontrivial branch which satisfies one of
the alternatives there. Recalling that $\beta=1/\nu$, we readily have the existence
of a nontrivial branch $\mathcal{C}$ in the $(\nu,v)$-space, satisfying all the
conditions in \eqref{mainsys}, with the possible exception of the positivity ones.
In view of Lemma \ref{lem:global_bound} we have that
\[
\mathcal{C}\subset \left\{(\nu,v) : \int_\Omega|\nabla v|^2 \leq \frac{C_1}{\nu},\
0<\nu \leq C_2 \right\}.
\]
Note that, in principle, $\mathcal{C}\cap\{\nu\geq \eps>0\}$ may be unbounded
in $C^{2,\alpha}$. Recalling that, in dimension $N\leq 3$, the nolinearities
in \eqref{mainsys} are Sobolev subcritical, by standard elliptic regularity
we have that $H^1$ bounds imply $C^{2,\alpha}$ ones, so that unboundedness
can happen only as $\nu\to0$.
The last thing that is left to prove, to complete the first part of the theorem, is that the branch we
obtained consists of componentwise positive pairs. This easily follows since, by the Maximum
Principle, if the pairs $(v_{1,n},v_{2,n})$ solve \eqref{mainsys}, with $\nu=\nu_n>0$ and
$\lambda_i=\lambda_{i,n}$, and
\[
v_{i,n}\to\bar v_i,\qquad \nu_{i,n}\to\bar \nu_i,\qquad \lambda_{i,n}\to\bar \lambda_i,
\]
then either $\bar\nu=0$ or $\bar\nu>0$ and $\bar v_1,\bar v_2$ are strictly positive in
$\overline{\Omega}$.

Coming to the second part, let $\mu^*$ be a simple positive Neumann eigenvalue. By Lemma \ref{lem:linear_G}
we have that $\partial^2_{\beta,v} G(\beta,0)=L$. In order to apply Theorem
\ref{thm:Rabinowitz_simple}, we only have to check the compatibility condition, which in
our case writes
\[
\left(-\sqrt{\alpha_1}\psi^*, \sqrt{\alpha_2}\psi^*\right) \not\in \Range(I-\beta^* L)
\]
(here $\psi^*$ is an eigenfunction associated to $\mu^*=2\sqrt{\alpha_1\alpha_2}\beta^*$).
By contradiction, let us assume that $(-\sqrt{\alpha_1}\psi^*, \sqrt{\alpha_2}\psi^*) =
(I-\beta^* L)w$, i.e.
\[
\begin{cases}
-\Delta w_i = - 2\beta^* \alpha_i w_j + (-1)^i  \mu^* \psi^* \sqrt{\alpha_i} & \text{in }\Omega\\
\int_\Omega w_i  = 0\\
\partial_n w_i =0 & \text{on }\partial\Omega.
\end{cases}
\]
Reasoning as in the proof of Lemma \ref{lem:linear_G}, it is easy to prove that $w=0$, and
hence $\psi^*=0$, a contradiction.
\end{proof}

\begin{rem}\label{rem:Taylor} If we suppose that $g_1$ and $g_2$ are smooth, then the branch $\mathcal{S}$ bifurcating from $(\beta^*,1,1)$ is a smooth curve (at least in a neighborhood of that point), and its parametrization can be made more precise. In order to simplify the following computations, we set
\[
(\beta, v) \in \mathcal{S} \Leftrightarrow 0 = \hat{F}(\beta, v, \lambda):= F(\beta, v, v, \beta \lambda)
\]
for some $\lambda \in \Rset^2$, where $F$ is as in \eqref{eq:Fdefn}. Then, $\hat{F} \from \R\times E \times \R^2\to C^{0,\alpha}(\overline{\Omega},\R^2) \times
\R^2$ is smooth and satisfies
\begin{align}
\hat{F}_{v} (\beta, v, \lambda) [w, \ell]  & = \left(- \Delta w_i + \beta (2g_i'(v_j^2)v_i v_j w_j + g_i(v_j^2)w_i - \ell_i v_i - \lambda_i w_i), 2 \int_\Omega v_i w_i \right) ,\\
\hat{F}_{v, \beta} (\beta, v, \lambda) [w, \ell] & = (2 g_i'(v_j^2)v_i v_j w_j + g_i(v_j^2) w_i - \ell_i v_i - \lambda_i w_i, 0), \label{Fuv} \\
\hat{F}_{v, v} (\beta, v, \lambda) [w, \ell; h, p] & = \left( \beta[4 g_i''(v_j^2)v_i v_j^2  + 2g_i'(v_j^2)v_i] w_j h_j \right. \notag \\ & \left. + 2 \beta g_i'(v_j^2)v_j w_j h_i + 2\beta g_i'(v_j^2)v_j w_i h_j - \beta \ell_i h_i - \beta p_i w_i, 2 \int_\Omega h_i w_i \right),  \label{Fuu} \\
\hat{F}_{v, v, v} (\beta, v, \lambda) [w, \ell; h, p; z, q] & = \left( \beta[8 g_i'''v_i v_j^3 + 8 g_i''v_i v_j +  4g_i''v_i v_j]w_j h_j z_j + \beta [4 g_i''v_j^2 + 2 \beta g_i' ]w_j h_j z_i \right. \notag \\
& + \beta[4 g_i''v_j^2 + 2 g_i' ]w_j h_i z_j +\beta [4g_i''v_j^2 + 2 g_i' ]w_i h_j z_j, 0).
\label{Fuuu}
\end{align}

If $(\beta^*,1,1, g_1(1),g_2(1))$ is a simple bifurcation point, then $\Ker(\hat{F}_{v})$ is spanned by the vector $V^* = (-\sqrt{\alpha_1} \psi^*, \sqrt{\alpha_2} \psi^*, 0, 0)$, and $R(\hat{F}_{v}) = \{(\Psi, \cdot) = 0 \}$, where $\Psi = (-\sqrt{\alpha_2} \psi^*, \sqrt{\alpha_1} \psi^*, 0, 0)$. Therefore, arguing as in \cite[Chap. 5]{MR1225101}, if we set
\[
A := (\Psi, \hat{F}_{v, \beta} [V^*]), \quad B := \frac{1}{2}(\Psi, \hat{F}_{v, v} [V^*, V^*]), \quad C := -\frac{1}{6A}(\Psi, \hat{F}_{v, v,v} ,[V^*]^3)
\]
where all the derivatives of $\hat{F}$ are evaluated at $(\beta^*,1,1, g_1(1),g_2(1))$, the following expansions hold true
\[
\beta = \beta^* - \frac{B}{A} \eps + o(\eps), \quad \text{(if $B \neq 0$)},
\]
and
\[
\begin{array}{ll}
v = (1,1) + \frac{A}{B} (\beta - \beta^*)\cdot(\sqrt{\alpha_1} \psi^*, -\sqrt{\alpha_2} \psi^*) + o(\beta - \beta^*) & \text{if $B \neq 0$}, \\
v = (1,1) \pm \left(\frac{\beta - \beta^*}{C}\right)^{1/2}\cdot(\sqrt{\alpha_1} \psi^*, -\sqrt{\alpha_2} \psi^*) + O(\beta - \beta^*) & \text{if $B = 0, C \neq 0$}.
\end{array}
\]
Note that in the latter case, if $C > 0$ (respectively, $C <0$) the bifurcating branch emanates on the right (respectively, left) of $\beta^*$. In our case, the coefficients $A,B,C$ have the explicit form
\begin{align*}
A & = - 4 g_1' g_2' \int (\psi^*)^2 < 0,\\
B & = \beta^*[2g_2''\sqrt{(g_1')^{3}} - 2g_1''\sqrt{(g_2')^{3}} + 3g_1'g_2'(\sqrt{g_1'} - \sqrt{g_2'})] \int (\psi^*)^3, \\
C & = \frac{\beta^*}{-6A}\left[12g_1'g_2'\sqrt{g_1'g_2'} + \sum_{i=1,2} (-8(g_j')^2 g_i''' + 12g_i''(g_j'\sqrt{g_i'g_j'} - (g_j')^2))\right] \int (\psi^*)^4,
\end{align*}
where all the derivatives of $g_i$ are evaluated at $s = 1$.

We observe that if $N=1$, the bifurcation is always \emph{critical},
namely $B = 0$, as every eigenfunction $\psi^*$ satisfies  $\int (\psi^*)^3 = 0$. In the variational case, where $g_i''(1) = g_i'''(1) = 0$, the bifurcating branch emanates on the right, namely $(\beta, v) \in \mathcal{S}$ is such that $ \beta \ge \beta^*$ (and therefore $\nu \le \nu^*$), at least in a neighborhood of $\beta^*$.
\end{rem}

\section{Classification of solutions in dimension \texorpdfstring{$N=1$}{N=1}}\label{sec:1d}

In this section we restrict our attention to the case $\Omega=(0,1)\subset\R$.
Consequently, in the following $(v_1,v_2)$ denotes a solution of the problem ($i,j=1,2$, $j\neq i$)
\begin{equation}\label{eq:1Dsys}
\begin{cases}
-\nu v_i'' + g_i(v_j^2) v_i = \lambda_i v_i & \text{in $(0,1)$} \\
\int_0^1 v_i^2\,dx = 1,\qquad v>0  & \text{in $[0,1]$} \\
v'_i(0) = v'_i (1) = 0 .
\end{cases}
\end{equation}
In particular, each $v_i$ is $C^2([0,1])$, and it has at least
an inflection point in $(0,1)$ (just apply Rolle's Theorem to $v'_i$).
Furthermore, $v''_i(x)$ has the same sign of $g_i(v_j^2(x))-\lambda_i$,
for every $x$.
\begin{lemma}\label{lem:opposite_at_endpoints}
$v_i$ and $v_j$ have opposite concavity at $0$ and $1$. More precisely:
\begin{itemize}
 \item $g_i(v_j^2(0))>\lambda_i \iff g_j(v_i^2(0))<\lambda_j$;
 \item $g_i(v_j^2(1))>\lambda_i \iff g_j(v_i^2(1))<\lambda_j$.
\end{itemize}
\end{lemma}
\begin{proof}
Let us assume, for instance, $g_i(v_j^2(0))>\lambda_i$ and, by contradiction,
$g_j(v_i^2(0))\geq\lambda_j$ (the other cases are analogous).

Then $v''_i(0)>0$, and there exists $\xi\in(0,1]$ such that
\begin{equation}\label{eq:contrinlem}
v''_i>0\text{ in }[0,\xi),\qquad v''_i(\xi)=0 \text{ (and hence
$g_i(v_j^2(\xi))=\lambda_i$).}
\end{equation}
Notice that $\xi<1$, otherwise $v_i$ would have no inflection point in $(0,1)$.
By convexity and monotonicity we deduce that
\[
x\in(0,\xi] \quad\implies\quad v_i(x)>v_i(0)\quad \implies\quad g_j(v_i^2(x))-\lambda_j
> g_j(v_i^2(0))-\lambda_j\geq 0.
\]
But then also $v_j$ is (convex and) increasing in $[0,\xi]$, so that
\[
g_i(v_j^2(\xi))>g_i(v_j^2(0))>\lambda_i,
\]
in contradiction with \eqref{eq:contrinlem}.
\end{proof}
Next we exploit standard uniqueness results for ODEs in order to detect a number
of situations in which a considered solution is the trivial one.
\begin{lemma}\label{lem:uniq_trivial}
Let one of the following condition hold:
\begin{enumerate}
 \item there exists $\xi\in[0,1]$ such that
\[
g_1(v_2^2(\xi))=\lambda_1,\qquad g_2(v_1^2(\xi))=\lambda_2,
\qquad v_1'(\xi)=v_2'(\xi)=0;
\]
 \item there exist $0\leq x_{1}< x_2\leq1$ such that, for some $i$,
 $v_i$ is constant in $I=[x_1,x_2]$;
 \item for some $i$, $g_i(v_j^2(0))=\lambda_i$;
 \item for some $i$, $g_i(v_j^2(1))=\lambda_i$;
\end{enumerate}
Then $(v_1,v_2)$ is the trivial solution.
\end{lemma}
\begin{proof}
Under the assumptions of case 1, uniqueness for the Cauchy problem
\[
\begin{cases}
-\nu v_i'' + g_i(v_j^2) v_i = \lambda_i v_i \qquad \text{in $(0,1)$} \\
v_i(\xi) = \sqrt{g_j^{-1}(\lambda_j)},\quad v_i'(\xi)=0,\qquad i=1,2,\,j\neq i
\end{cases}
\]
implies that both $v_1$ and $v_2$ are constant, and we can conclude exploiting
the normalization in $L^2(0,1)$.

If 2 holds, the equation for $v_i$ implies $g_i(v_j^2(\xi))=\lambda_i$ on
$I$. But then also $v_j$ is constant in $I$, forcing $g_j(v_i^2(\xi))=\lambda_j$.
Since both $v_i'$ and $v'_j$ are identically zero in $I$, case 1. applies.

Recalling the Neumann boundary conditions, also cases 3 and 4 can be reduced to 1:
indeed, by Lemma \ref{lem:opposite_at_endpoints}, $g_i(v_j^2)-\lambda_i$ vanishes
at one endpoint if and only if $g_j(v_i^2)-\lambda_j$ does.
\end{proof}
The following key lemma asserts that between two consecutive maxima of each $v_i$ there
exists an interval of concavity of $v_j$.
\begin{lemma}\label{lem:key1D}
Let $0\leq x_{1}< x_2\leq1$ be such that, for some $i$,
\[
v'_i(x_1) = v'_i(x_2)=0,\qquad v''_i(x_1)\leq0,\, v''_i(x_2)\leq0.
\]
Then either $(v_1,v_2)$ is the trivial solution, or there exists $\xi\in(x_1,x_2)$
such that
\[
g_j(v_i^2(\xi))<\lambda_j.
\]
Analogously, if $v_i'$ vanishes and $v''_i$ is
nonnegative at $x_1,x_2$ then $g_j(v_i^2(\xi))>\lambda_j$ for some $\xi
\in(x_1,x_2)$.
\end{lemma}
\begin{proof}
We have to show that, in case $g_j(v_i^2(x))\geq\lambda_j$ for every
$x\in[x_1,x_2]$, then $(v_1,v_2)$ is the trivial solution. Under such assumption
we have that
\[
\begin{cases}
v_j''\geq0 & \text{in }(x_1,x_2)\\
g_i(v_j^2)\leq\lambda_i & \text{at }\{x_1,x_2\},
\end{cases}
\qquad\text{so that }g_i(v_j^2)\leq\lambda_i\text{ in the whole }[x_1,x_2].
\]
Thus $v_i''\leq 0$ in $[x_1,x_2]$. Since $v_i'=0$ at $x_1$ and $x_2$, we obtain
that $v_i$ is constant in $[x_1,x_2]$, and Lemma \ref{lem:uniq_trivial} (case 2)
applies, concluding the proof.
\end{proof}
The above result provides a sharp control on the critical and inflection
points of each $v_i$, as we show in the next sequence of lemmas.
\begin{lemma}\label{lem:isolated_crit}
If $(v_1,v_2)$ is non trivial then both components have only isolated critical points.
\end{lemma}
\begin{proof}
Let by contradiction $\xi\in[0,1]$ be an accumulation point for the set of critical
points of $v_i$. Of course
\[
v_i'(\xi)=0.
\]
We recall that, for any pair of critical points $x_1<x_2$, if both $v_i''(x_1)>0$
and $v_i''(x_2)>0$ then there exists a third critical point $y_1\in(x_1,x_2)$ such
that $v_i''(y_1)\leq0$ (and the same holds for opposite inequalities).
Using this fact, it is not difficult to construct two sequences $x_n\to\xi$,
$y_n\to\xi$ such that
\[
v_i'(x_n)=v_i'(y_n)=0, \qquad v_i''(x_n)\geq0,\,v_i''(y_n)\leq0.
\]
Applying repeatedly Lemma \ref{lem:key1D} we deduce the existence of sequences
$\overline{\xi}_n\to\xi$, $\underline{\xi}_n\to\xi$ such that
$g_j(v_i^2(\underline{\xi}_n))<\lambda_j<g_j(v_i^2(\overline{\xi}_n))$.
This promptly yields
\[
g_j(v_i^2(\xi))=\lambda_j.
\]
Now back to the sequence $(x_n)$, applying Rolle's Theorem we first deduce the
existence of a sequence $z_n\to\xi$ such that $0= v_i''(z_n) = g_i(v_j^2(z_n))
-\lambda_i$, implying
\[
g_i(v_j^2(\xi))=\lambda_i,
\]
and then of a sequence $z'_n\to\xi$ with
\[
0=v_j'(z'_n)\to v_j'(\xi).
\]
Summing up, we are in a position to apply Lemma \ref{lem:uniq_trivial} (Case 1),
obtaining that $(v_1,v_2)$ is trivial, a contradiction.
\end{proof}
\begin{lemma}\label{lem:nondeg_extrema}
Let $x_0\in[0,1]$ be a point of local minimum for $v_i$. Then either $(v_1,v_2)$
is the trivial solution, or
\[
g_j(v_i^2(x_0))<\lambda_j,\qquad g_i(v_j^2(x_0))>\lambda_i
\]
(in particular, it is non degenerate). An analogous statement (with reverse inequalities) holds for local maxima.
\end{lemma}
\begin{proof}
If $x_0=0$ or $x_0=1$, then the statement is a consequence of Lemma
\ref{lem:opposite_at_endpoints}. Otherwise, since $x_0$ is an isolated critical
point, it is a strict minimum, and the following points are well defined:
\[
x_1=\inf\{x\in[0,x_0) : v'_i<0\text{ in }(x,x_0)\},\qquad
x_2=\sup\{x\in(x_0,1] : v'_i>0\text{ in }(x_0,x)\}.
\]
Then $x_1,x_2$ satisfy the assumptions of Lemma \ref{lem:key1D}, providing
the existence of $\xi\in(x_1,x_2)$ such that
\[
g_j(v_i^2(x_0))=\min_{[x_1,x_2]}g_j(v_i^2)\leq
g_j(v_i^2(\xi))<\lambda_j,
\]
which is the first inequality required.

On the other hand, since $x_0$ is an isolated strict minimum we have that
$g_i(v_j^2(x_0))\geq\lambda_i$ in a neighborhood of $x_0$. Since the last
inequality implies that $v_j$ is strictly concave in a neighborhood of $x_0$,
we deduce also the second (strict) inequality.
\end{proof}
\begin{lemma}\label{lem:nondeg_critical}
If $(v_1,v_2)$ is not the trivial solution, then any critical point of each
component is non degenerate.
\end{lemma}
\begin{proof}
Let us assume by contradiction that $\xi$ is a degenerate critical point of $v_i$.
By Lemmas \ref{lem:isolated_crit}, \ref{lem:nondeg_extrema} we have that
$\xi\in(0,1)$ is an isolate inflection point. Therefore
\[
v_i'(\xi)=0, \qquad g_i(v_j^2(\xi))=\lambda_i,
\]
and $\xi$ is a local extremum for $v_j$. But then Lemma \ref{lem:nondeg_extrema}
applies again, implying that either $g_i(v_j^2(\xi))>\lambda_i$ or
$g_i(v_j^2(\xi))<\lambda_i$, a contradiction.
\end{proof}
\begin{lemma}\label{lem:1flessoogni2estremi}
Let $(v_1,v_2)$ be non trivial, and $x_1<x_2$ be such that, for some $i$,
\[
v_i'(x_1)=v_i'(x_2)=0, \qquad v_i'>0\text{ in }(x_1,x_2).
\]
Then both $v_i$ and $v_j$ have exactly one inflection point in $[x_1,x_2]$.
An analogous statement holds for the opposite monotonicity.
\end{lemma}
\begin{proof}
By Lemma \ref{lem:nondeg_extrema} we immediately deduce the existence
of $\xi\in(x_1,x_2)$ such that
\begin{equation}\label{eq:flessi_extr1}
v_j''<0\text{ in }[x_1,\xi),\qquad v_j''>0\text{ in }(\xi,x_2],
\end{equation}
and $v_j$ has exactly one inflection point in $[x_1,x_2]$.

On the other hand, the inflection points of $v_i$ are the solutions of
\begin{equation}\label{eq:flessi_extr2}
g_j(v_i^2(x))=\lambda_j,\qquad x\in[x_1,x_2].
\end{equation}
Since $g_j(v_i^2(x_1))>\lambda_j$ and $g_j(v_i^2(x_2))<\lambda_j$ (and again by Lemma \ref{lem:nondeg_extrema}),
equation \eqref{eq:flessi_extr2} has an odd number of solutions. On the other hand, taking into account
\eqref{eq:flessi_extr1}, equation \eqref{eq:flessi_extr2} has at most one solution in $[x_1,\xi]$ and one in $[\xi,x_2]$, respectively.
\end{proof}
Collecting the previous results we have the following characterization of nontrivial solutions.
\begin{proposition}\label{prop:degree}
If $(v_1,v_2)$ is not the trivial solution, then there exists $k\in\N$ such that
both $v_1$ and $v_2$ have exactly $k$ critical points, all non degenerate, and $k+1$
isolated inflection points in $(0,1)$.
\end{proposition}
\begin{proof}
Let $n_i$ denote the number of critical points of $v_i$ in $(0,1)$ (they are well defined by Lemma \ref{lem:isolated_crit}), and $m_i$ the number of
inflection points. By Lemma \ref{lem:1flessoogni2estremi} we have that $n_i + 2 = m_i = m_j$, and the claim follows.
\end{proof}
We are ready to conclude the proof of the main result of this section.
\begin{proof}[Proof of Theorem \ref{thm:intro_1D} (first part)]
First of all let $\mathcal{C}\subset\mathcal{S}$ be a continuum of nontrivial solutions,
and
\[
\mathcal{C} \ni (\nu_n,v_{1,n},v_{2,n}) \to (\bar\nu,\bar v_{1},\bar v_{2}),
\qquad\text{as }n\to+\infty.
\]
Using Proposition \ref{prop:degree}, it is not difficult to prove that, if the number
of interior critical points of each $v_{i,n}$ is constant, equal to $k$, then
\begin{itemize}
 \item either $\bar\nu=0$;
 \item or $\bar\nu>0$ and $(\bar v_1,\bar v_2)$ is the trivial solution;
 \item or $\bar\nu>0$ and each $\bar v_i$ has exactly $k$ interior critical points.
\end{itemize}

Now recall that, being $N=1$ and $\Omega = (0,1)$, the Neumann eigenvalues and eigenfunction
of $-\partial^2_{xx}$ have the form
\[
\mu_k = (k \pi)^2,\quad \psi_k = A \cos (k\pi\,x)\ (A \neq0),\qquad k\in\N,
\]
and every eigenvalue $\mu_k$ is simple. Applying Theorem \ref{thm:intro_bif} we have the
existence, for every $k\geq1$, of continua $\mathcal{C}_k \subset \mathcal{S}$ which consist,
locally near $(\mu_k,1,1)$, of pairs having exactly $k-1$ critical points (by the local
parameterization, because also $\psi_k$ has exactly $k-1$ critical points). The initial
argument tells that each $\mathcal{C}_k$ is characterized by the number of critical points of its
components, so that two of them cannot meet, and each of them is unbounded in the sense of Theorem \ref{thm:intro_bif} (since we are in dimension $N=1\leq3$).

We are only left to prove segregation: this is the object of the next section.
\end{proof}

\section{Segregation in dimension \texorpdfstring{$N=1$}{N=1}}\label{sec:1dsegr}

As we already mentioned (see Remark \ref{rem:notalwayssegregate}), we can not expect that all
arbitrary families of nontrivial
solutions segregate. Nonetheless, restricting our attention to $\mathcal{C}_k$ as
in Theorem \ref{thm:intro_1D}, for some fixed $k$, we can obtain more precise results.

In the following, we focus on ${(\nu_n,v_{1,n},v_{2,n})} \subset \mathcal{C}_k$, a sequence of
solutions of \eqref{mainsys}, with $\nu=\nu_n>0$ and $\lambda_i=\lambda_{i,n}$, whose
components have exactly $k-1$ critical points in $(0,1)$, all non-degenerate.
For easier notation, we will drop the subscript $n$ throughout the proofs,
except when some confusion may arise; in particular, properties of
\[
v_1,v_2,\lambda_1,\lambda_2\qquad\text{ as }\nu \rightarrow 0,
\]
are those of the considered sequence, when $\nu_n\to0$ as $n \rightarrow +\infty$.

As a first step, we want to rule out the possibility that the branch ``collapses'' to the trivial solution as $\nu \rightarrow 0$.

\begin{prop}\label{noncollapse}
Suppose that
\[
v_{1,n}\to 1,\qquad v_{2,n}\to 1,\qquad \nu_n\to\bar \nu,
\]
where the convergence is uniform in $[0,1]$. Then, $\lambda_{i,n} \to g_i(1)$ and
$\bar \nu > 0$.
\end{prop}

\begin{proof} The proof will be carried out in three steps, and considering the system solved by $u_i := v_i - 1$, which is
\begin{equation}\label{usys}
\left\{ \begin{array}{ll}
-\nu u''_1 = (\lambda_1 - G_1(1 + u_2))(1 + u_1), & \text{in $(0,1)$} \\
-\nu u''_2 = (\lambda_2 - G_2(1 + u_1))(1 + u_2), & \\
u'_1 = u'_2 = 0 \quad \text{at $\{0,1\}$}, \\
\int_0^1 (1+u_1)^2 = \int_0^1 (1+u_2)^2 = 1,
\end{array} \right.
\end{equation}
where we have set $G_i(t) := g_i(t^2)$ for all $t \ge 0$. Note that $u_i \to 0$ uniformly in $[0,1]$.

Without loss of generality, we set ($\bar x = \bar x_n$)
\begin{equation}\label{Mdef}
M := \max_{i = 1,2, \, x \in [0,1]} |u_i(x)| = u_1(\bar x).
\end{equation}

\underline{Step 1.} $|\lambda_i - G_i(1)|/{M} \rightarrow 0$ as $n \rightarrow \infty$. Indeed, note first that $\int_0^1 (1+u_i)^2 = 1$ implies that
\begin{equation}\label{ui_av}
\int_0^1 u_i = -\frac{1}{2}  \int_0^1 u^2_i.
\end{equation}
Moreover, a Taylor expansion in the equations of \eqref{usys} gives
\[
-\nu u''_i = \left(\lambda_i - G_i(1) - G'_i(1)u_j - \frac{G''_i(\xi)}{2}u^2_j\right)(1 + u_i),
\]
where $\xi$ is a bounded function in $(0,1)$ (uniformly with respect to $n$). By integrating the equation and using the boundary conditions we obtain
\[
(\lambda_i - G_i(1)) \int_0^1 (1 + u_i) = G'_i(1)\int_0^1 u_j  + G'_i(1)\int_0^1 u_i \, u_j + \int_0^1\frac{G''_i(\xi)}{2}u^2_j(1 + u_i).
\]
Hence, using \eqref{ui_av},
\[
|\lambda_i - G_i(1)| \int_0^1 (1 + u_i) \le \frac{G'_i(1)}{2} \int_0^1 u^2_j + G'_i(1)\int_0^1 |u_i| \, |u_j| + \int_0^1\frac{G''_i(\xi)}{2}u^2_j(1 + u_i),
\]
which leads to the assertion, as $\int_0^1 (1 + u_i) \to 1$, $|u_i|, |u_j| \le M$ in $[0,1]$ and $M \to 0$. The first conclusion of the {proposition} also follows, as $G_i(1) = g_i(1)$.

\underline{Step 2.} Assume by contradiction that $\bar \nu = 0$. We proceed with a blow-up
analysis, setting
\[
\tilde{u}_i(x) = \frac{1}{M} u_i \left(\sqrt{ \nu } \, x +\bar x\right) \qquad \forall x \in \left(-\frac{\bar x}{\sqrt{\nu}}, \frac{1-\bar x}{\sqrt{\nu}}\right) =: \widetilde{\Omega}_n.
\]
We have that $|\tilde{u}_i| \le 1$ in $\widetilde{\Omega}_n$ and $\tilde{u}_1(0) = 1$. Moreover, $\tilde{u}_i$ solves
\[
- \tilde{u}''_i = \left(\frac{\lambda_i - G_i(1)}{M} - (G'_i(1) + o(1))\tilde{u}_j \right)(1 + u_i) \qquad \text{in $\widetilde{\Omega}_n$}.
\]
Note that (up to subsequences)
\[
\widetilde{\Omega}_n \to \widetilde{\Omega}_\infty:=
\begin{cases}
[\bar X,+\infty) & \text{if }-\frac{\bar x}{\sqrt{\nu}} \to \bar X<+\infty\\
(-\infty,\bar X] & \text{if }\frac{1-\bar x}{\sqrt{\nu}} \to \bar X<+\infty\\
\R & \text{otherwise.}
\end{cases}
\]
Using the equation (twice) and the uniform boundedness of $\tilde{u}_i$ on
$\widetilde{\Omega}_n$, we argue that $\tilde{u}'''_i$ is bounded on compact subsets of $[0,
\infty)$, uniformly as $n \to \infty$. Hence, $\tilde{u}_i \to \widetilde{U}_i \in
C^2(\widetilde{\Omega}_\infty)$ locally in $C^{2,\alpha}$ where $\widetilde{U}_i$ has at most
$k$ intervals of monotonicity and solve, in $\widetilde{\Omega}_\infty$,
\begin{equation}\label{limitsys}
\begin{cases}
\widetilde{U}''_1 = G'_1(1) \widetilde{U}_2 \\
\widetilde{U}''_2 = G'_2(1) \widetilde{U}_1,
\end{cases}
\end{equation}
in view of the conclusion of Step 1. Note that, in case $\widetilde{\Omega}_\infty\neq\R$,
we can use the Neumann conditions in order to extend $\widetilde{U}_i$ by even reflection
around $\bar X$, in such a way that $\widetilde{U}_1, \widetilde{U}_2$ solve \eqref{limitsys} in the whole $\Rset$.

\underline{Step 3.} To reach a contradiction we are going to show that system \eqref{limitsys}
does not admit nontrivial bounded solutions having a finite number of
oscillations (recall that $\widetilde U_1(0)=1$). We can reason as in Section \ref{sec:bifurc}, setting
\[
W_\pm = \sqrt{\alpha_2} \widetilde U_1 \pm \sqrt{\alpha_1} \widetilde U_2,
\]
and obtaining a decoupled system:
\[
\begin{cases}
W_+'' = \saOaT \, W_+, \\
W_-'' = -\saOaT \, W_-. \\
\end{cases}
\]
Therefore, since $W_+$ is bounded it must be constant and identically zero. We deduce
that $\widetilde U_1$, $\widetilde U_2$ are proportional, so that
\[
\widetilde U_1'' = -\saOaT \, \widetilde U_1,
\]
which forces $\widetilde U_1 \equiv 0$ (since it has at most $2k$ monotonicity intervals in
$\R$).
\end{proof}
Next we turn to the case in which $\|v_i\|_{L^\infty}$ is uniformly bounded along the sequence,
for both $i$. To treat such case we need the following Liouville-type result.
\begin{lem}\label{liouville2}
Let $V_i \in C^2(\R)$, $0\leq V_i \leq M$, $\Lambda_i \ge 0$ be such that
\[
-V_i'' = (\Lambda_i - g_i(V_j^2)) V_i \qquad \text{ in }\R.
\]
If both $V_i$ have at most a finite number of monotonicity intervals, then one of the following holds:
\begin{enumerate}
 \item either $V_1 \equiv 0$, $\Lambda_2 = 0$,
 \item or $V_2 \equiv 0$, $\Lambda_1 = 0$,
 \item or $V_1 \equiv V_2 \equiv 0$,
 \item or $g_i(V_j^2) \equiv \Lambda_i$, $i=1,2$.
\end{enumerate}
\end{lem}

\begin{proof}
First of all, we can reason as in Lemma \ref{lem:uniq_trivial} to show that,
if some $V_i$ is constant in an interval, then $(V_1,V_2)$ is constant in $\R$, and as
a consequence we always fall in one of the above cases. Secondly, assume that
some $\Lambda_i=0$: then $V_i$ is constant, and again the lemma follows
by elementary considerations.

We are left to deal with the case $\Lambda_1,\Lambda_2>0$ and $V_1,V_2$ non constant
and strictly positive. Since
both $V_i$ have a finite number of monotonicity intervals, the equations imply that they
also have a finite number of inflection points (and they have at least one, since they are bounded in $\R$). We deduce the existence of $a\in\R$ such that, say,
\[
V_1',V_2',V_1'',V_2''\text{ do not change sign in }(a,+\infty).
\]
In particular, the limits $V_i(+\infty)$ exist and $V_i'(+\infty)=V_i''(+\infty)=0$.

Assume that $V_i'\geq0$ for $x>a$, so that $V_i''\leq0$ in the same interval.
Then $V_i(+\infty)>0$, and
\[
g_i(V_j^2 (x) ) \leq \Lambda_i\text{ in }(a,+\infty),\qquad g_i(V_j^2 (+\infty) ) =
\Lambda_i  - \frac{V_i''(+\infty)}{V_i(+\infty)} = \Lambda_i,
\]
so that also $V'_j$ is non negative.

Now we can lower $a$ in such a way that, say,
\[
V_1''(a)=0.
\]

If $V_1'\leq0$ for $x>a$, we deduce that also $V_2'\leq0$ in the same interval. But then $\Lambda_1 - g_1(V_2^2)$ is increasing, and $V_1''\leq0$ for $x>a$,
a contradiction since $V_1$ is decreasing and bounded.

On the other hand, let $V_1'\geq 0$ for $x>a$. Then $V_1(+\infty)>0$, and
\[
g_1(V_2^2 (a) ) = \Lambda_1 = g_1(V_2^2 (+\infty) ) - \frac{V_1''(+\infty)}{V_1(+\infty)}
= g_1(V_2^2 (+\infty) ),
\]
forcing $V_2$ to be constant in $[a,+\infty)$, again a contradiction.
\end{proof}
Using the previous result, we can show that uniform $L^\infty$ bounds imply segregation.
\begin{lem}\label{l1l2zero} Assume that $\|v_{i,n}\|_\infty \le C$ for both
$i$. Then, up to subsequences,
\[
\lambda_{i,n} \rightarrow 0,  \quad i=1,2,
\]
as $\nu_n \rightarrow 0$.
\end{lem}
\begin{proof} Let us assume by contradiction that, for instance, $\lambda_1\not\to0$. We choose a sequence
$(x_{1,n})_n\subset[0,1]$ such that (omitting the subscript $n$)
\[
v_1(x_1):= \max_{[0,1]} v_1\geq1,\qquad \text{and put }\tilde{v}_i(x) := v_i\left(x_1 + x\sqrt{\nu}\right).
\]
Then, $\tilde{v}_i$ solves
\[
-\tilde{v}''_i(x) = \left(\lambda_i - {g_i(\tilde{v}_j^2(x))}\right) \tilde{v}_i(x)
\quad \text{in $(-x_1\nu^{-1/2},(1-x_1)\nu^{-1/2})$},
\]
$\|\tilde{v}_i\|_\infty \le C$, $\tilde{v}_i \ge 0$. The equations and \eqref{identity1} guarantee local $C^3$ boundedness of $\tilde{v}_i$, thus, up to subsequences,
$\tilde{v}_i \rightarrow V_i$ locally in $C^2$. Moreover, $\lambda_i \rightarrow \Lambda_i \ge 0$. We argue that, possibly up to an even extension, each $V_i$ has at most
$2k$ intervals of monotonicity and
\[
-V_i'' = (\Lambda_i - g_i(V_j^2)) V_i  \quad \text{in
$\R$}.
\]
Then, Lemma \ref{liouville2} applies, but since $V_1(0) \ge 1$ and
$\Lambda_1 > 0$, we deduce that
\begin{equation}\label{cases1}
g_i(v_j^2(x_1)) - \lambda_i \rightarrow 0\text{ for both $i$},
\end{equation}
and also $\lambda_2 \not\to 0$ (as $g_2(v_1^2(x_1))\geq g_2(1)>0$).
We can implement the same argument using
\[
v_1(x_2):= \min_{[0,1]}v_1\leq1,\qquad \text{and }\tilde{w}_i(x) := w_i\left(x_2 + x\sqrt{\nu}\right).
\]
Passing to the limit (we keep the same sequence $\lambda_i \rightarrow \Lambda_i>0$ as before),
and recalling Lemma \ref{lem:nondeg_extrema}, we have that $W_2(0)>0$ and then
\begin{equation}\label{cases2}
g_i(v_j^2(x_2)) - \lambda_i \rightarrow 0\text{ for both $i$}.
\end{equation}
Combining \eqref{cases1} and \eqref{cases2}, we deduce that $v_1\to 1$ uniformly on
$[0,1]$.

Now, since also $\Lambda_2>0$, we can exchange the role of $v_1$ and $v_2$, obtaining that $v_2\to 1$ too, in contradiction with Proposition \ref{noncollapse}.
\end{proof}
We are left to deal with the case of $\max_{[0,1]} (v_{1,n}+v_{2,n}) \rightarrow +\infty$, namely when $v_1$ or $v_2$ are not bounded uniformly  in $n$. To treat
this situation we need to exploit the finite number of maxima of each component along
$\mathcal{C}_k$, as enlighten in the following lemma (for convenience we write explicitely
the dependence on $n$).
\begin{lemma}\label{lem:intervallino}
Let $\max_{[0,1]} (v_{1,n}+v_{2,n}) \rightarrow +\infty$. There exist and index $i$,
constants $C,\delta>0$ (independent of $n$), and a sequence of intervals $I_n\subset[0,1]$ such that, up to subsequences:
\[
\begin{split}
& |I_n| = \delta\\
& \max_{I_n} v_{i,n} = \max_{\partial I_n} v_{i,n} \to +\infty\\
& \max_{I_n} v_{j,n} \leq C.
\end{split}
\]
\end{lemma}
\begin{proof}
Let
\[
Z_n := \left\{ z\in[0,1] : z\text{ is a local maximum for }v_{i,n}\text{, for some }i,\text{ and }
v_{i,n}(z) \to +\infty\right\}.
\]
Since we are considering elements of $\mathcal{C}_k$, we have that
\[
Z_n = \{z_{l,n}\}_{l=1,\dots,h},\qquad\text z_{1,n}<\dots<z_{h,n},\qquad h\leq k+1
\]
(recall that, by Lemma \ref{lem:nondeg_extrema}, if $z$ is a local maximum for $v_i$ then $g_i(v_j^2(z))\leq \lambda_i$). Up to subsequences, we can assume that each $z_{l,n}$ is
a maximum for some $v_{i,n}$, with $i$ independent of $n$; furthermore we can assume that,
for each $l$, $z_{l,n}\to z_l\in[0,1]$. We distinguish three cases.

\underline{Case 1}: for some $l$, $z_{l}<z_{l+1}$.  We choose $i$ so that $z_{l,n}$ is
a local maximum for $v_{i,n}$ and
\[
2\delta = z_{l+1} - z_{l},\qquad I_n=[z_{l,n},z_{l,n}+\delta].
\]
By construction, neither $v_{i,n}$ nor $v_{j,n}$ can have interior maxima which go to infinity;
therefore the required properties for $\max_{I_n} v_{i,n}$ follow from the fact that
$v_{i,n}(z_{l,n})\to+\infty$, while those for $\max_{I_n} v_{j,n}$ descend again by Lemma
\ref{lem:nondeg_extrema}.

\underline{Case 2}: $z_{1}=\dots=z_{h} \neq 1$. One can reason as above, by choosing $i$ so
that $z_{h,n}$ is a local maximum for $v_{i,n}$ and
\[
2\delta = 1 - z_{h},\qquad I_n=[z_{h,n},z_{h,n}+\delta].
\]

\underline{Case 3}: $z_{1}=\dots=z_{h} \neq 0$. We can choose $i$ so that $z_{1,n}$ is a
local maximum for $v_{i,n}$ and
\[
2\delta = z_{1},\qquad I_n=[z_{1,n}-\delta,z_{1,n}].      \qedhere
\]
\end{proof}
The last tool we need is the following standard comparison lemma.
%
%
\begin{lem}[{\cite[Lemma 4.4]{MR2146353}}]\label{comparison} Suppose that $u \in C^2(a,b) \cap C([a,b])$ satisfies
\[
- u''(x) \le - M u(x), \quad 0 \leq u(x) \leq A,\qquad
\text{in $(a,b)$}
\]
for some $A,M > 0$. Then, for every $0<\delta<(b-a)/2$,
\[
 u(x) \le 2A\,e^{-\delta \sqrt{M}} \qquad
 \text{in $[a+\delta,b-\delta]$}.
\]
\end{lem}
\begin{proof} By comparison with
the solution of $-w'' = -Mw$ in $(a,b)$, $w(a)=w(b)=A$.
\end{proof}
\begin{remark}\label{rem:comparison on half interval}
By even reflection, we have that if $u$ is as in Lemma \ref{comparison} and furthermore $u'(a)=0$, then the estimate holds on
any $[a,b']\subset[a,b)$, choosing $\delta=b-b'$.
\end{remark}
We are in a position to prove that segregation occurs
also when some $v_i$ is unbounded, thus completing the proof
of Theorem \ref{thm:intro_1D}.
\begin{lemma}\label{lgoestozero} Let $\max_{[0,1]} (v_{1,n}+v_{2,n}) \rightarrow +\infty$.
Then (up to subs.) $\lambda_{i,n} \rightarrow 0$, for some $i$ (and
the corresponding $v_{i,n}$ is not uniformly bounded).
\end{lemma}
\begin{proof}
Let $i$, $I_n=:[z_n,z_n+\delta]$ be as in Lemma \ref{lem:intervallino}. We
can assume, w.l.o.g.,
\[
\max_{I_n} v_{i,n} = v_{i,n}(z_n)\to+\infty.
\]
We define the blow-up sequences
\[
\begin{split}
\tilde{v}_{i,n}(x) &:= \frac{1}{v_{i,n}(z_n)} v_{i,n}(z_n + x\sqrt{\nu_n})\\
\tilde{v}_{j,n}(x) &:=  v_{j,n}(z_n + x\sqrt{\nu_n}).
\end{split}
\]
Then, $\tilde{v}_{i,n}=\tilde{v}_i$ solves
\[
- \tilde{v}''_i = (\lambda_i - g_i(\tilde v_j^2)) \tilde{v}_i
\]
in $(0, \delta\nu^{-1/2})$, $0 \le \tilde{v}_i \le 1$ and $\tilde{v}_i(0) = 1$. Also $\lambda_i - g_i(\tilde v_j^2)$ is uniformly bounded in $[0,
\delta\nu^{-1/2}]$, by Lemmas \ref{lem:prel_estimates} and
\ref{lem:intervallino}.
Since both $\tilde v_i$ and $\tilde v_i''$ are uniformly bounded on compact
sets, we deduce that also $\tilde v_i'$ is bounded, and there exists
$V\in C^1([0,+\infty))$ such that $v_i\to V$ in $C^1([a,b])$,
for every $[a,b]\subset [0,+\infty)$.

We claim that, if $V>0$ in $[a',b']\subset(0,+\infty)$, then
$\tilde v_j \rightarrow 0$ uniformly in $[a',b']$. Indeed, let $(a,b)
\supset [a',b']$ be such that $V\geq \eta>0$ in $(a,b)$. We deduce that,
in such interval,
\[
-\tilde v_j'' = (\lambda_j - g_j(v^2_{i}(z)\tilde v_i^2)) \tilde{v}_j\leq
\left(\lambda_j - C_g^{-1} v^2_{i}(z) \frac12 V_i^2\right)\tilde{v}_j
\leq - C^2  v^2_{i}(z) \tilde{v}_j,
\]
where $C>0$ depends on $\eta$ and $C_g$. Lemma \ref{comparison} applies, yielding
\[
0\leq \tilde{v}_j \leq C_1 e^{-C_2 v_i(z)}\to0 \qquad \text{in }[a',b'],
\]
as $C_2>0$ and $v_i(z)\to+\infty$.

Now, let $\lambda_{i}\to\Lambda\geq0$. We can pass to the limit in the
equation of $\tilde v_i$, deducing that
\[
\begin{cases}
 V\in C^1([0,+\infty)),\qquad 0\leq V\leq 1,\\
 V>0 \implies -V''=\Lambda V\\
 V(0)=1.
\end{cases}
\]
Let $[0,a)$, $a\leq+\infty$, be the maximal interval containing $0$ in which
$V>0$. If $a<+\infty$, by convexity we obtain that $V(a)=0$ and $V'(a)<0$,
a contradiction since $V(x)$ must be non negative also for $x>a$. Therefore $a=+\infty$ and $V$ is a bounded, concave function
on $\R^+$, i.e. $V\equiv 1$ and $\Lambda=0$.
\end{proof}

\begin{proof}[End of the proof of Theorem \ref{thm:intro_1D}]
Taking into account Corollary \ref{coro:suff_cond_segr}, the last part of
the theorem follows from Lemmas \ref{l1l2zero} and \ref{lgoestozero}.
\end{proof}

\section{Further properties of the first branch}\label{sec:firstbranch}

To conclude, we complete the analysis started in Section \ref{sec:1dsegr}
by restricting our attention to the first bifurcation branch
$\mathcal{C}_1$. Since $k=1$, such branch consists
of monotone solutions,
and for concreteness we assume that the sequence we are considering is such
that $v_{1,n}$ is decreasing and $v_{2,n}$ is increasing (and $\nu_n
\rightarrow 0$ as $n \rightarrow \infty$). As before, we will omit the
subscript $n$, when no confusion arises. We denote by $\xi_{1,n}, \xi_{2,n}
\in (0,1)$ the unique inflection points of the considered pair:
\[
- v'_{1,n}(\xi_{1,n}) = \max_{[0,1]} |v'_{1,n}(x)|, \quad
v'_{2,n}(\xi_{2,n}) = \max_{[0,1]} |v'_{2,n}(x)|.
\]
A number of (rather elementary) a priori estimates can be deduced from the
monotonicity of the components. We collect them in the following three
lemmas.
\begin{lem} The following inequalities hold
\begin{align}
&v_1^2(x) \le \frac{1}{x} \quad \forall x > 0, \quad  \quad v_2^2(x) \le \frac{1}{1-x} \quad \forall x < 1, \label{bound12abv}\\
&\xi_1[v_1^2(0) + v_1(0)v_1(\xi_1) + v_1^2(\xi_1)] \le 3, \label{x1ineq} \\
&(1-\xi_2)[v_2^2(1) + v_2(1)v_2(\xi_2) + v_2^2(\xi_2)]  \le 3, \label{x2ineq} \\
&|v_1'(x)| (x-x_0)  \le x_0^{-1/2}  \qquad \forall x_0  \ge \xi_1, x \in [x_0, 1], \label{v1pineq} \\
&|v_2'(x)| (x_0-x)  \le (1-x_0)^{-1/2}  \qquad \forall x_0  \le \xi_2, x \in [0, x_0]. \label{v2pineq}
\end{align}
\end{lem}
\begin{proof} Estimates \eqref{bound12abv} follow by the $L^2$ constraint:
\[
1 \ge \int_0^x v_1^2 \ge x \, v_1^2(x), \quad 1 \ge \int_x^1 v_2^2 \ge (1-x) v_2^2(x).
\]
For the other estimates, it is crucial to observe that $\lambda_1 -
g_1(v_2^2(\xi_1)) = 0$, as $v_1''(\xi_1) = 0$ ($\xi_1$ is a point in $(0,1)$
where $v_1'$ achieves its minimum). The function $\lambda_1 - g_1(v_2^2)$ is
decreasing, so by the equation for $v_1$ in \eqref{mainsys} we deduce that
$v_1$ is concave on $[0,\xi_1]$ and convex on $[\xi_1,1]$.

Concavity implies that
\[
v_1(x) \ge v_1(0) + \xi_1^{-1}(v_1(\xi_1) - v_1(0)) x
\]
in $[0,\xi_1]$. By invoking the $L^2$ constraint of $v_1$ and integrating,
\[
1 \ge \int_0^{\xi_1} v_1^2(x) dx \ge \frac{\xi_1}{3} [v_1^2(0) + v_1(0)v_1(\xi_1) + v_1^2(\xi_1)],
\]
and \eqref{x1ineq} follows. Similarly, concavity of $v_2$ on $[\xi_2,1]$ produces \eqref{x2ineq}.

By convexity of $v_1$ on $[x_0,1]$, and \eqref{bound12abv},
\[
-v_1'(x)(x-x_0) \le -v_1'(x)(x-x_0) + v_1(x) \le v_1(x_0) \le x_0^{-1/2},
\]
for all $x \in [x_0,1]$, and we have \eqref{v1pineq}. Similarly, \eqref{v2pineq} follows by concavity of $v_2$ on $[0,x_0]$.
\end{proof}
\begin{lem}\label{boundslemma} Suppose that $\|v_{1,n}\|_\infty \le C_1$.
Then,
\begin{equation}\label{bound1bel}
v^2_{1,n}(x) \ge \frac{1}{2} \quad \text{in $[0,a_1]$},
\end{equation}
where $a_1 = a_1(C_1)$. Similarly, if $\|v_{2,n}\|_\infty \le C_2$,
\begin{equation}\label{bound2bel}
v^2_{2,n}(x) \ge \frac{1}{2} \quad \text{in $[a_2,1]$},
\end{equation}
where $a_2 = a_2(C_2)$.
\end{lem}
\begin{proof} In view of the $L^2$ constraint on $v_1$ and its monotonicity we have that
\[
1 = \int_0^x v_1^2 + \int_x^1 v_1^2 \le x \, v_1^2(0) + (1-x) v_1^2(x)
\]
for all $x \in [0,1]$. Therefore, if $a_1 = (2C_1^2 - 1)^{-1}$,
\[
v_1^2(a_1) \ge \frac{1- a_1 v_1^2(0)}{1-a_1} \ge \frac{1- a_1 C_1^2}{1-a_1} = \frac{1}{2}.
\]
The assertion for $v_1$ follows. The estimate \eqref{bound2bel} for $v_2$ is analogous.
\end{proof}
\begin{lem} For both $i$ it holds
\begin{equation}\label{inftybound}
|v^2_{i,n}(0) - v^2_{i,n}(1)| \le 2 \frac{\lambda_{i,n}}{\nu},
\end{equation}
\begin{equation}\label{energyest}
\nu \|v'_{i,n}\|_\infty^2 \le \lambda_{i,n} \|v_{i,n}\|_\infty^2.
\end{equation}
\end{lem}
\begin{proof} We will prove the assertion when $i = 1$, the argument is analogous when $i = 2$. Multiplying the equation for $v_1$ by $v_1$ and integrating on $[0,x]$ yields
\[
-\nu v_1'(x) v_1(x) + \nu\int_0^x (v_1')^2 = \int_0^x (\lambda_1 - g_1)v_1^2,
\]
thus
\[
-\frac{\nu}{2} (v_1^2 )'(x) = -\nu v_1'(x) v_1(x)  \le  \lambda_1.
\]
By integrating again on $[0,1]$ we obtain \eqref{inftybound}.

On the other hand, testing the equation for $v_1$ by $v_1'$ and integrating
on $[0,\xi_1]$ we obtain
\[
\frac{\nu}{2} v_1'(\xi_1)^2 = \frac{\lambda_1}{2} v_1(0)^2 +
\int_0^{\xi_1} g_1v_1v_1',
\]
and \eqref{energyest} follows since $v_1'\leq 0$ in $[0,1]$.
\end{proof}

After the above preliminary estimates, the first part of our analysis is devoted to show that $\mathcal{C}_1$ enjoys uniform $L^\infty$ bounds as $\nu\to0$. To this aim we need two preliminary lemmas.
\begin{lem}\label{lnubound} Suppose that, for some $i$, $\|v_{i,n}\|_\infty \le C$ and $\lambda_{j,n} \rightarrow 0$. Then, there exists $C' > 0$ that does not depend on $n$ such that
\[
\lambda_{i,n} \le C' \nu_n.
\]
\end{lem}
\begin{proof} We will detail the proof in the case $i=1$. Note that
\[
g_2(v_1^2(x)) - \lambda_2 \ge g_2(1/2) - \lambda_2 \ge C_g^{-1}/2 - \lambda_2 \ge C_g^{-1}/4 \quad \text{in $[0,a_1]$}
\]
by the monotonicity of $g_2$, \eqref{bound1bel}, \eqref{Gass} and $\lambda_2 \rightarrow 0$. Hence,
\[
-v_2'' = -\frac{g_2(v_1^2)- \lambda_2}{\nu}\, v_2 \le \frac{C_g^{-1}}{4 \nu} \, v_2
\]
in $(0,a_1)$, and Lemma \ref{comparison} (or better Remark \ref{rem:comparison on half interval}) allows to conclude that
\begin{equation}\label{est1}
v_2(x) \le 2v_2(a_1) e^{-C/\sqrt{\nu}}\quad  \text{in $[0,a_1/2]$},
\end{equation}
for some $C = C(a_1, C_g^{-1}) > 0$.

Recalling Definition \ref{def:Nash},
we choose ${w}(x) := \sqrt{\frac{4}{a_1}} \cos\left(\frac{\pi}{a_1}x\right)$ for $x \in [0, a_1/2]$ and ${w} \equiv 0$ in $[a_1/2, 1]$ to conclude that, for some $C' > 0$,
\[
\lambda_1 \le \int_0^{a_1/2} \nu({w}')^2 +  g_1(v_2^2){w}^2 \le \frac{\nu \pi}{a_1} + g_1\left(4v_2^2(a_1) e^{-2C/\sqrt{\nu}}\right) 
\le\frac{ \nu \pi}{a_1} + \frac{4C_g e^{-2C/\sqrt{\nu}}}{1-a_1} \le C' \nu,
\]
by \eqref{Gass}, \eqref{est1} and \eqref{bound12abv}.
\end{proof}
\begin{lem}\label{v1pl2} Suppose that, for some $i$, $\|v_{i,n}\|_\infty  \rightarrow +\infty$. Then,
\begin{equation}
\nu \|v_{i,n}'\|^2_\infty \le C \lambda_{j,n}
\end{equation}
for some $C > 0$ that does not depend on $n$.
\end{lem}

\begin{proof}  We will detail the proof in the case $i=1$, thus assuming
\[
v_1(0) \to +\infty.
\]
Note that $|v_2'| \le c_2$ in $[0,1/2]$ for some $c_2 > 0$. Indeed, if $v_2$ is bounded then Lemma \ref{lgoestozero} and Lemma \ref{lnubound} imply that $\lambda_2 \le C_2' \nu$ for some $C_2' > 0$, and by \eqref{energyest} it follows that $\|v_2'\|^2_\infty \le C_2' \|v_2\|^2_\infty$. On the other hand, if $v_2(1)$ is unbounded, then $\xi_2 \rightarrow 1$ (see \eqref{x2ineq}), and then we have the required bound by \eqref{v2pineq} (choose, for example, $x_0 =3/4$).

We now integrate the equation for $v_2$ on $[\xi_1, 1/2]$, use \eqref{Gass} and $\int v_2^2 = 1$ to obtain
\begin{equation}\label{bbound}
C_g^{-1} \int_{\xi_1}^{1/2} v_1^2 v_2 \le \int_{\xi_1}^{1/2} g_2(v_1^2) v_2 = \lambda_2 \int_{\xi_1}^{1/2} v_2 + \nu (v_2'(1/2)- v_2'(\xi_1)) \le  \lambda_2 + 2c_2 \nu
\end{equation}

Let $\overline{T}_1$ be the function
\[
\overline{T}_1(x) = \nu (v'_1(x))^2 + [\lambda_1 - g_1(v_2^2(x))]v_1^2(x) + 2 \int_{1/2}^x g_1'(v_2^2(\sigma))v'_2(\sigma) v_2(\sigma)v_1^2(\sigma) d\sigma.
\]
$\overline{T}_1$ is easily verified to be constant in $[0,1]$.
Since $\lambda_1 - g_1(v_2^2(x))$ is decreasing and $\lambda_1 - g_1(v_2^2(\xi_1)) = 0$,
$\lambda_1 - g_1(v_2^2(1/2)) \le 0$, as $\xi_1 \le 1/2$ ($\xi_1 \rightarrow 0$ because $v_1(0) \rightarrow +\infty$). Hence,
\begin{multline*}
\nu (v'_1(\xi_1))^2 + 2 \int_{1/2}^{\xi_1} g_1'(v_2^2)v'_2 v_2 v_1^2 d\sigma = \overline{T}_1(\xi_1) \\ = \overline{T}_1(1/2) = \nu (v'_1(1/2))^2 + [\lambda_1 - g_1(v_2^2(1/2))]v_1^2(1/2)
\end{multline*}
and
\[
\nu \|v'_1\|^2_\infty = \nu (v'_1(\xi_1))^2 \le \nu (v'_1(1/2))^2 + 2 \int^{1/2}_{\xi_1} g_1'(v_2^2)v'_2  v_1^2 v_2 d\sigma \le C (\nu + \lambda_2).
\]
The last bound comes from $|v_2'| \le c_2$, $|v_2| \le 1+ c_2/2$, \eqref{bbound} and $|v_1'| \le c_1$ in $[1/2, 1]$ (use \eqref{v1pineq}: $v_1$ is unbounded and $\xi_1 \rightarrow 0$).
\end{proof}
As already mentioned, the previous results allow to obtain uniform bounds for the sequence we are considering.
\begin{lemma}\label{vboundinfty} There exists $C_\infty > 0$, that does not depend on $n$, such that
\[
\|v_{i,n}\|_\infty \le C_\infty, \qquad i=1,2.
\]
\end{lemma}

\begin{proof} Without loss of generality, we can assume by contradiction that
\begin{equation*}
v_1(0) \rightarrow \infty, \quad \text{and $\lambda_2 \le \lambda_1$.}
\end{equation*}
Indeed, if both $v_1(0)$ and $v_2(1)$ are unbounded, such condition can be guaranteed by interchanging the role of $v_1$ and $v_2$. Otherwise, suppose that, say, $v_1(0) \rightarrow \infty$ and $v_2$ is bounded: by Lemmas \ref{lgoestozero} and \ref{lnubound} there exists $C > 0$ such that $\lambda_2 \le C \nu$, while $\lambda_1/ \nu \rightarrow \infty$ (otherwise $v_1$ would be bounded in view of \eqref{inftybound}). Therefore,
$\lambda_2 \le \lambda_1$ whenever $\nu$ is sufficiently small and we infer, by Lemma \ref{v1pl2}, the existence of $C > 0$ such that
\begin{equation}\label{eq2}
\|v_1'\|_\infty \le C \sqrt{\frac{\lambda_1}{\nu}}.
\end{equation}

We proceed as in the proof of Lemma \ref{lgoestozero}, by defining the blow-up sequences
\[
\tilde{v}_{1}(x) := \frac{1}{v_{1}(0)} v_{1}\left(x\sqrt{\frac{\nu}{\lambda_1}}\right),\qquad
\tilde{v}_{2}(x) :=  v_{2}\left(x\sqrt{\frac{\nu}{\lambda_1}}\right),
\]
Note that $0 \le \tilde{v}_1 \le 1$ in $[0, \lambda_1^{1/2}\nu^{-1/2}]$, and that $\tilde{v}_i(0) = 1$. Since, in such interval,
\[
|\tilde v_1'(x)| = \frac{1}{v_{1}(0)}\sqrt{\frac{\nu}{\lambda_1}}\left| v'_{1}\left(x\sqrt{\frac{\nu}{\lambda_1}}\right)\right| \leq \frac{C}{v_{1}(0)}\to 0
\]
(we used \eqref{eq2}), we deduce that $\tilde v_1 \to V\equiv 1$, uniformly in every $[a,b]\subset [0,+\infty)$. As a consequence, in any such interval,
\[
-\tilde v_2'' = \left(\frac{\lambda_2}{\lambda_1} - \frac{g_2(v_1(0)\tilde v_1^2)}{\lambda_1}\right) \tilde{v}_2 \leq
\left(1 - \frac{C_g^{-1}}{2\lambda_1} v^2_{1}(0)  \right)\tilde{v}_2
\leq - C^2 \frac{v^2_{1}(0)}{\lambda_1} \tilde{v}_2,
\]
with $C>0$, and Remark \ref{rem:comparison on half interval} applies, yielding
\[
\tilde v_2(x) \leq \tilde v_2(b+1) e^{-Cv_1(0)/\sqrt{\lambda_1}} \leq 2
e^{-Cv_1(0)/\sqrt{\lambda_1}} \qquad \text{ for }x\in[0,b],
\]
for $\nu$ sufficiently small (recall \eqref{bound12abv}). Then
\[
\frac{g_1(\tilde v_2^2)}{\lambda_1} \le
\frac{C_1 e^{-C_2v_1(0)/\sqrt{\lambda_1}}}{\lambda_1} \le
\frac{C_3}{v_1^2(0)} \to 0,
\]
and we can plug such estimate in the equation for $\tilde v_1$
\[
- \tilde{v}''_1 = \left(1 - \frac{g_1(\tilde v_2^2)}{\lambda_1}\right) \tilde{v}_1,
\]
in order to pass to the limit and obtain
\[
-V''=V \qquad \text{in }(0,+\infty),
\]
in contradiction with the fact that $V\equiv 1$.
\end{proof}
Uniform $L^\infty$ bounds readily provide Lipschitz ones, thus yielding convergence to some limiting profiles.
\begin{proposition}\label{vprimebound}
There exists $C'_\infty > 0$, not depending on $n$, such that
\[
\|v'_{i,n}\|_\infty \le C'_\infty \quad i=1,2.
\]
As a consequence, up to subsequences,
\begin{equation}\label{V1V2zero}
v_{i,n} \rightarrow V_i \text{ in $C^{0,\alpha}([0,1])$}, \qquad \text{ with }\int_0^1V_1^2 = \int_0^1 V_2^2=1\text{ and } V_1 \cdot V_2 \equiv 0 \text{ in }[0,1],
\end{equation}
and
\begin{equation}\label{lambdavsnu}
\frac{\lambda_ {i,n}}{\nu_n} \rightarrow \ell_i > 0,
\end{equation}
as $n \rightarrow +\infty$.
\end{proposition}
\begin{proof} Lemma \ref{vboundinfty} guarantees the uniform $L^\infty$ bound for $v_1, v_2$, hence $\lambda_1, \lambda_2 \rightarrow 0$ by Lemma \ref{l1l2zero}. As a consequence we can apply
Lemma \ref{lnubound}, for both $i$, obtaining that there exists $C'_i > 0$ that does not depend on $\nu$ such that
\[
\lambda_i \le C'_i \nu.
\]
This implies  that both $\nu \|v'_i\|_\infty^2 \le \nu C'_i \|v_1\|_\infty^2$, by \eqref{energyest}, and, up to subsequences, both $v_{i}\to V_i$ in $C^{0.\alpha}$ and $\lambda_i/\nu\to\ell_i\geq0$. Since
uniform convergence implies $L^2$-one, the required properties for the limiting profiles $V_i$ follow (recall Corollary \ref{coro:suff_cond_segr}), and the only thing that remains to be proved is that
both $\ell_i>0$.

Assume by contradiction that, for instance, $\ell_1=0$. Then we can use equation \eqref{inftybound} to infer that $V_1\equiv 1$,
in contradiction with \eqref{V1V2zero}.
\end{proof}
\begin{remark}\label{rem:H1conv}
Once we know that $v_{i,n}\to V_i$ uniformly, the strong $H^1$ convergence follows by standard arguments. Indeed, integrating the equations we have
\[
0\leq\frac{1}{\nu}\int_0^1 g_i(v_{j,n}^2)v_{i,n}\,dx = \frac{\lambda_{i,n}}{\nu}\int_0^1 v_{i,n}\,dx \leq C;
\]
therefore, testing with $v_{i,n}- V_i$ we infer
\[
\int_0^1 v'_{i,n} (v'_{i,n}- V'_i)\,dx \leq \max_{[0,1]}|v_{i,n}- V_i| \cdot \frac{1}{\nu}\int_0^1 (\lambda_{i,n}+g_i(v_{j,n}^2))v_{i,n}\,dx \to 0.
\]
As a consequence, weak $H^1$ convergence implies convergence in norm, and finally strong $H^1$ one.
\end{remark}
The remaining part of the section will be devoted to fully characterize the limits $V_i, \ell_i$. To this aim, we need a sharper analysis of the convergence of $v_{i,n}$.
\begin{lem}\label{comparison2} Suppose that, as $n\to+\infty$, $v_{1,n} (y_n)
\ge c \nu_n^{1/2 - \epsilon}$ for some $y_n \in [0,1)$, $c > 0$,
$0 < \epsilon \le 1/2$. Then there exists $c_1 > 0$ such that
\begin{equation}
v_{2,n}(x) \le 2v_{2,n}(y_n) e^{- c_1 (y_n-x) \nu_n^{-\epsilon}} \quad \text{in $[0,y_n]$}.
\end{equation}
\end{lem}

\begin{proof} By the monotonicity of $v_1$, \eqref{Gass} and \eqref{lambdavsnu},
\[
g_2(v_1^2(x)) - \lambda_2 \ge C_g^{-1} v_1^2(x) - \lambda_2 \ge \frac{C_g^{-1} c^2}{2} \nu^{1 - 2\epsilon} \quad \text{in $[0,y]$}
\]
as $\nu \rightarrow 0$. Hence,
\[
-v_2'' = -\frac{g_2(v_1^2)- \lambda_2}{\nu} v_2 \le -\frac{C_g^{-1} c^2}{2} \nu^{- 2\epsilon}\, v_2
\]
in $(0,y)$, and we can conclude using Remark
\ref{rem:comparison on half interval}.
\end{proof}

\begin{rem}\label{comparison3}
A direct consequence of the previous lemma, which will be used thoroughly in
the sequel, is that if $\liminf_{\nu \rightarrow 0} v_1 (y) > 0$ for some
$y \in [0,1)$, then there exists $c_2 > 0$, $y < b < 1$ (that does not depend
on $\nu$) such that
\begin{equation}\label{comparison30}
v_2(x) \le C_\infty e^{- \frac{c_2}{\sqrt{\nu}}} \quad \text{in $[0,b]$}.
\end{equation}
Indeed, the assumption guarantees that $v_1 (y) \ge 2 c > 0$ for some $c >
0$, so, by Proposition \ref{vprimebound}, $v_1(y') \ge c$ for some $y' > y$.
Hence,
\[
v_2(x) \le C_\infty e^{- c_1 (y'-x) \nu^{-1/2}} \quad \text{in $[0,y']$},
\]
that implies \eqref{comparison30} if we choose $y < b < y'$, and $c_2 =
c_2(c_1, b, y, y') > 0$.

Note that $v_1(0) \ge 1$ for all $\nu$ (otherwise the mass constraint $
\int_0^1 v_1^2 dx = 1$ would be violated), thus
\begin{equation}\label{comparison31}
v_2(0) \le C_\infty e^{- c_2 \nu^{-1/2}} = o(\nu^a) \quad \text{for all $a > 0$.}
\end{equation}
for some $c_2 > 0$.

Analogous conclusions hold if $v_1$ and $v_2$ are interchanged.
\end{rem}

\begin{lem} The limit $V_i, \ell_i$ satisfy, in $[0,1]$,
\begin{align}
V_1(x) & = \frac{2}{\sqrt{\pi}} \sqrt[4]{\ell_1}  \cos\left(\sqrt{\ell_1}x\right) \cdot \chi_{\left[0, \frac{\pi}{2\sqrt{\ell_1}}\right]}(x), \label{V1} \\
V_2(x) & = \frac{2}{\sqrt{\pi}} \sqrt[4]{\ell_2}  \cos\left(\sqrt{\ell_2}\left(x-1\right)\right) \cdot \chi_{\left[1-\frac{\pi}{2\sqrt{\ell_2}}, 1\right]}(x). \label{V2}
\end{align}
Moreover, as $n\rightarrow +\infty$,
\begin{equation}\label{xilimit}
\xi_{1,n} \rightarrow \frac{\pi}{2\sqrt{\ell_1}}, \quad \xi_{2,n} \rightarrow 1-\frac{\pi}{2\sqrt{\ell_2}}.
\end{equation}
\end{lem}

\begin{proof}
Let $x_1 > 0$ be such that $[0, x_1) = \{x: V_1(x) > 0\}$ ($V_1$ is
identically zero in $[x_1,1]$). If $y < x_1$, $v_1(y)$ is bounded away from
zero, uniformly with respect to $\nu$, hence $v_2(x) \le C_\infty
e^{- \frac{c_2}{\sqrt{\nu}}}$ in $[0,y]$ by \eqref{comparison30}.
Therefore, $g_1(v_2^2) = o(\nu)$ uniformly in $[0,y]$, that is
\[
\frac{\lambda_1 - g_1(v_2^2)}{\nu} = o(1)
\]
uniformly on compact subsets of $[0, x_1)$. Hence, we might pass to the limit (weakly) into the equation for $v_1$: let $\varphi$ be a smooth test function, with support laying in $[0, x_1)$. The equation reads
\[
\int_0^{x_1} v_1' \varphi' dx = \int_0^{x_1} \frac{\lambda_1 - g_1(v_2^2(x))}{\nu} v_1 \varphi dx,
\]
and passing to the limit (Proposition \ref{vprimebound} ensures weak convergence in $H^1((0,1))$ of $v_i$ to $V_i$),
\[
-V_1'' = \ell_1 V_1 \quad \text{in $(0,x_1)$},
\]
$V_1'(0) = 0$ and $V_1(x_1)=0$. Thus, being $V_1$ positive, it has to be of the form $A \cos\left(\sqrt{\ell_1}x\right)$ in $(0,x_1)$, for some $A > 0$. This forces $x_1 = \pi/(2\sqrt{\ell_1})$. Moreover, $\int_0^1 V_1^2 = 1$, since by uniform convergence the $L^2$-constraint passes to the limit,
and $A$ must satisfy $A=\frac{2}{\sqrt{\pi}} \sqrt[4]{\ell_1}$. The characterization of $V_2$ is analogous.

As for the second assertion, we argue that $v_1(\xi_1) \rightarrow 0$. If not, $v_2(\xi_1) \le C_\infty e^{- \frac{c_2}{\sqrt{\nu}}} = o(\nu^{1/2})$ by \eqref{comparison30}, that is not compatible with $g_1(v_2^2(\xi_1)) = \lambda_1 \ge c_1 \nu$. Hence, $\lim \xi_1 \ge x_1$. Suppose that $\lim \xi_1 > x_1$; note that $v_1$ is concave on $(0,\xi_1)$, so $v_1(x) \ge v_1(0) + (v_1(\xi_1) - v_1(0))x/\xi_1$ in $[0,\xi_1]$. We infer
\[
\lim v_1(x_1) \ge v_1(0) \left(1 - \lim \frac{x_1}{\xi_1}\right) > 0,
\]
which contradicts $v_1(x_1) \rightarrow V_1(x_1) = 0$. Then, $\xi_1 \rightarrow x_1 = \pi/(2\sqrt{\ell_1})$.
\end{proof}

\begin{lem}\label{smallintegrals} For all $a > 2$ it holds true that
\[
\int_0^1 v_{i,n}^a v_{j,n}^2 \,dx = o(\nu_n), \qquad\text{as }n\to+\infty.
\]
\end{lem}
\begin{proof} The assertion will be proved for $i=1$, the other case being completely analogous. Let us define $y \in (0,1)$ as the unique point such that
\[
v_1(y) = \nu^{1/6}.
\]
The point $y$ is well defined and bounded
away from zero because $v_1$ is uniformly strictly positive in a neighborhood
of $x=0$ and $v_1(1) = o(\nu)$ (by \eqref{comparison31} for $v_1$). By Lemma
\ref{comparison2} we also have $v_2(x) \le C e^{- c_1 (y-x)
\nu^{-1/3}}$ in $[0,y]$. We split the interval $[0,1]$ into two
subintervals, and exploit $\|v_i\|_\infty \le C_\infty$.

In $[0, y - \nu^{1/6}]$, by monotonicity we have
$v_2(x) \le C e^{- c_1 \nu^{-1/6}}$ and
\[
\int_0^{y - \nu^{1/6}} v_1^a v_2^2 \, dx \le C e^{- 2c_1
\nu^{-1/6}} = o(\nu).
\]

In $[y - \nu^{1/6}, 1]$, recalling that $\int_0^1 v_1^2 v_2^2
\le C_g^{-1} \lambda_1 \le C\nu$ by \eqref{Gass} and \eqref{lambdavsnu},
we obtain
\[
\int_{y - \nu^{1/6}}^1 v_1^a v_2^2 dx \leq
v_1^{a-2}\left(y - \nu^{1/6}\right) \cdot \int_0^1 v_1^2 v_2^2 \leq
v_1^{a-2}\left(y - \nu^{1/6}\right) \cdot C\nu = o(\nu),
\]
as $v_1\left(y - \nu^{1/6}\right)$ goes to zero (recall that $v_1$ converges uniformly, and that $v_1(y)\to0$).
\end{proof}
The last part of our analysis focuses on the ``interface'' between $v_1$ and
$v_2$, namely we are going to consider the point $x_m = x_{m,n} \in (0,1)$ such that
\[
m_n = v_1(x_{m,n}) = v_2(x_{m,n}).
\]
We follow ideas introduced in \cite{MR3021546}, to treat the one-dimensional
variational case.
Note that by strict monotonicity of $v_{i,n}$, $x_{m,n} \in (0,1)$ is
well-defined, and
\[
m_n \rightarrow 0, \quad x_{m,n} \rightarrow x_0 \in (0,1),
\]
in view of \eqref{V1V2zero} and the fact that $v_{1,n}$ and $v_{2,n}$ are
bounded away from zero in neighborhoods of $x=0$ and $x=1$ respectively (see
Lemma \ref{boundslemma}).

In what follows, we will write
\[
g_i(s) = \gamma_i s + h_i(s), \quad \text{for all $s \ge 0$},
\]
where $\gamma_i =g'_i(0) > 0$, $h_i(0)=0$, $h_i'(0) = 0$.
\begin{rem}\label{hexpansion} The functions $h_i$ have to be considered as ``lower order terms'' in the vanishing viscosity limit, and we will usually Taylor expand them around $s=0$, namely
\[
h_i(v^2_j(x)) = a_i(x) v^4_j(x), \quad h'_i(v^2_j) = b_i(x) v^2_j(x),
\]
where $|a_i(x)|, |b_i(x)| \le C$ for some universal constant $C > 0$ (depending on $g''$).
\end{rem}
The ``joint energy" is going to be crucial in our analysis:
\begin{multline}
T(x) := \frac{1}{\gamma_1}\left[\nu (v'_1(x))^2 + [\lambda_1 -
h_1(v_2^2(x))]v_1^2(x) + 2 \int_{x_m}^x h_1'(v_2^2(\sigma))v'_2(\sigma)
v_2(\sigma)v_1^2(\sigma) d\sigma\right] \\ +
\frac{1}{\gamma_2}\left[\nu (v'_2(x))^2 + [\lambda_2 - h_2(v_1^2(x))]v_2^2(x)
- 2 \int_x^{x_m} h_2'(v_1^2(\sigma))v'_1(\sigma) v_1(\sigma)v_2^2(\sigma) d
\sigma\right] \\ - v_1^2(x) v_2^2(x).
\end{multline}
Of course, along any pair $(v_{1,n},v_{2,n})$, $T_n(x)=T(x)$ is constant
(indeed $T_n'(x)\equiv0$).
\begin{lem} It holds
\begin{equation}\label{Tvalue}
\frac{\lambda_{1,n}}{\gamma_1} v_{1,n}^2(0) + o(\nu_n) = T_n =
\frac{\lambda_{2,n}}{\gamma_2} v_{2,n}^2(1) + o(\nu_n),
\end{equation}
as $n \rightarrow +\infty$.
\end{lem}

\begin{proof} Note firstly that
\begin{equation}\label{terminibrutti}
\left| \int_{x_m}^x h_i'(v_j^2 )v'_j v_j v_i^2 d\sigma \right| \le \int_0^1 |b_i v'_j| v_j^3 v_i^2 d\sigma \le C\int_0^1 v_j^3 v_i^2 = o(\nu),
\end{equation}
by Lemma \ref{smallintegrals} and Remark \ref{hexpansion}. Note also that $v_2^2(0) = o(\nu)$ by \eqref{comparison31}.

Therefore, being $v_1$ bounded by $C_\infty$,
\[
\begin{split}
T(0) = & \frac{\lambda_1}{\gamma_1} v_1^2(0) - \frac{ v_1^2(0) }{\gamma_1}h_1(v_2^2(0)) + \frac{2}{\gamma_1}  \int_{x_m}^0 h_1'(v_2^2)v'_2 v_2v_1^2 d\sigma \\ & +
\frac{\lambda_2 - h_2(v_1^2(0))}{\gamma_2}v_2^2(0)  - \frac{2}{\gamma_2} \int_0^{x_m} h_2'(v_1^2)v'_1 v_1v_2^2 d\sigma - v_1^2(0)v_2^2(0) =
\frac{\lambda_1}{\gamma_1} v_1^2(0) + o(\nu).
\end{split}
\]
Similarly,
\[
T(1) = \frac{\lambda_2}{\gamma_2} v_2^2(1) + o(\nu). \qedhere
\]
\end{proof}
\begin{lem}\label{m4nufinite} It holds true that
\[
\limsup_{n \rightarrow +\infty}\frac{m_n^4}{\nu_n} < +\infty.
\]
\end{lem}
\begin{proof} Arguing by contradiction,
\[
\frac{m^4}{\nu} \rightarrow \infty,
\]
possibly along a subsequence. Let $\tilde{v}_i(x) := \frac{1}{m} v_i\left(x_m + x \frac{\sqrt{\nu}}{m}\right)$. Then, $\tilde{v}_i$ solves
\[
- \tilde{v}''_i = \left(\frac{\lambda_i}{m^2} - \gamma_i \tilde{v}_j^2 - \frac{h_i(m^2 \tilde{v}_j^2)}{m^2}\right) \tilde{v}_i, \quad \text{in $I_\nu = \left(-\frac{m^2}{\nu}x_m, (1-x_m)\frac{m^2}{\nu}\right).$}
\]
Note that $I_\nu$ tends to the whole real line as $\nu \rightarrow 0$ ($x_m$ is bounded away from $x=0$ and $x=1$, and $m^{2}\nu^{-1} \rightarrow \infty$), $\tilde{v}_i(0) = 1$ and
\[
|\tilde{v}_i(y) - \tilde{v}_i(0)| \le |y| \|v_i'\|_\infty \frac{\sqrt{\nu}}{m^2} \rightarrow 0, \quad \text{for all $y \in [a,b] \subset I_\nu$}.
\]
Thus, $\tilde{v}_i$ converges uniformly on compact subsets of $I_\nu$, as $\|v_i'\|_\infty$ is bounded by $C'_\infty$. Moreover, $\lambda_i m^{-2} \rightarrow 0$ (by \eqref{lambdavsnu}) and $h_i(m^2 \tilde{v}_j^2){m^{-2}} \rightarrow 0$ uniformly on compact subsets of $I_\nu$ (see Remark \eqref{hexpansion}). Hence,
\begin{equation}\label{tildevtoone}
\tilde{v}_i \rightarrow 1, \quad \tilde{v}'_i \rightarrow 0 \quad \text{locally uniformly.}
\end{equation}

We then have
\begin{multline}\label{Ttilde1} \sum_{i=1,2}\frac{1}{\gamma_i}\left[(\tilde{v}'_i(0))^2 + \left(\frac{\lambda_i}{m^2} - \frac{h_i(m^2\tilde{v}_j^2(0))}{m^2}\right)\tilde{v}_i^2(0)\right]
 - \tilde{v}_1^2(0) \tilde{v}_2^2(0) = \\ \frac{T(x_m)}{m^4} = \frac{\lambda_1}{\gamma_1 m^4} v_1^2(0) + o(\nu/m^4) \ge \frac{c_1 \nu}{\gamma_1 m^4} v_1^2(0) + o(\nu/m^4) \ge 0
\end{multline}
by \eqref{Tvalue} and \eqref{lambdavsnu} when $\nu$ is close enough to zero. On the other hand, the left hand side of \eqref{Ttilde1} goes to $-1$ as $\nu \rightarrow 0$ by \eqref{tildevtoone}, a contradiction.
\end{proof}
\begin{lemma} As $n \rightarrow +\infty$, there exists $L > 0$ such that
\begin{equation}\label{mnu}
\frac{m_n^4}{\nu_n} \rightarrow L.
\end{equation}
Moreover,
\begin{equation}\label{vpbelow}
\liminf |v'_{i,n}(x_{m,n})| > 0.
\end{equation}
\end{lemma}
\begin{proof}  Let us assume by contradiction that
\[
\frac{m^4}{\nu} \rightarrow 0.
\]
Let $\tilde{v}_i(x) := \frac{1}{m} v_i\left(x_m + m \, x\right)$. Then, $\tilde{v}_i$ solves
\[
- \tilde{v}''_i = \frac{m^4}{\nu}\left(\frac{\lambda_i}{m^2} - \gamma_i \tilde{v}_j^2 - \frac{h_i(m^2 \tilde{v}_j^2)}{m^2}\right) \tilde{v}_i, \quad \text{in $I_\nu = \left(-\frac{x_m}{m}, \frac{1-x_m}{m}\right).$}
\]
Note that $I_\nu$ tends to the whole real line as $\nu \rightarrow 0$, $\tilde{v}_i(0) = 1$ and $\|\tilde{v}_i' \|_\infty= \|{v}_i' \|_\infty \le C'_\infty$, so $\tilde{v}_i$ converges uniformly on compact subsets of $I_\nu$. Moreover, $\lambda_i m^{-2} \rightarrow 0$ and $h_i(m^2 \tilde{v}_j^2){m^{-2}} \rightarrow 0$ uniformly on compact subsets of $I_\nu$ (as in the proof of the Lemma \ref{m4nufinite}). Hence, $\tilde{v}_i \rightarrow W_i$ locally in $C^2(\Rset)$, and $W_i'' = 0$. Since $W_i$ is positive we have
\begin{equation}\label{tildevtoone2}
W_i \equiv 1 \quad \text{in $\Rset$, $i=1,2$.}
\end{equation}
Therefore,
\begin{multline}\label{Ttilde2} \sum_{i=1,2}\frac{1}{\gamma_i}\left[(\tilde{v}'_i(0))^2 + \frac{m^4}{\nu} \left(\frac{\lambda_i}{m^2} - \frac{h_i(m^2\tilde{v}_j^2(0))}{m^2}\right)\tilde{v}_i^2(0)\right]
 - \frac{m^4}{\nu} \tilde{v}_1^2(0) \tilde{v}_2^2(0) = \\ \frac{T(x_m)}{\nu} = \frac{\lambda_1}{\gamma_1 \nu} v_1^2(0) + o(1) \ge \frac{c_1}{\gamma_1} v_1^2(0) + o(1) > 0
\end{multline}
by \eqref{Tvalue} and \eqref{lambdavsnu} when $\nu \rightarrow 0$. However, the left hand side of \eqref{Ttilde1} goes to zero as $\nu \rightarrow 0$ by \eqref{tildevtoone2}, that is not possible. Hence, $L > 0$.

To prove \eqref{vpbelow} we proceed as before, setting
$\tilde{v}_i(x) := \frac{1}{m} v_i\left(x_m + m \, x\right)$. We have that
$\tilde{v}_i \rightarrow W_i$ locally in $C^2(\Rset)$, and $(W_1,W_2)$ solve
\[
\begin{cases}
W_1'' = L \gamma_1 W_2^2 W_1, \\
W_2'' = L \gamma_2 W_1^2 W_2.
\end{cases}
\]
in $\Rset$. $W_1$ and $W_2$ are also also positive and monotone, so $W_i'(0) \neq 0$. We conclude by observing that $|v_i'(x_m)| = |\tilde{v}'_i(0)| \rightarrow |W_i'(0)| > 0$.
\end{proof}

\begin{lem} As $n \rightarrow +\infty$, it holds true that
\[
\xi_{1,n} \le x_{m,n} \le \xi_{2,n},
\]
and
\begin{equation}\label{xilimit2}
\xi_{1,n} \rightarrow x_0, \quad \xi_{2,n} \rightarrow x_0.
\end{equation}
\end{lem}

\begin{proof} Since $\xi_1$ is the inflection point of $v_1$ we have
\[
C_g^{-1} v_2^2(\xi_1) \le g_1(v_2^2(\xi_1)) = \lambda_1 \le C'_1 \nu,
\]
also by invoking \eqref{Gass} and \eqref{lambdavsnu}. Hence, $v_2(\xi_1) \le \sqrt{(C_g C'_1)\nu}$, but $v_2(x_m) = m \sim \sqrt[4]{L \nu}$ by \eqref{mnu}, so $ v_2(\xi_1) \le v_2(x_m)$ for $\nu$ sufficiently small. Monotonicity of $v_2$ implies that $\xi_1 \le x_m$, while $x_m \le \xi_2$ is obtained by an analogous argument at the inflection point $\xi_2$ of $v_2$.

Suppose now that $\xi_2 - \xi_1 = 4\eta$ and $\eta$ is uniformly bounded away from zero as $\nu \rightarrow 0$. Assume, without loss of generality, that $x_m \in [\xi_1, \xi_1 + 2\eta]$ (on the other hand, if $x_m \in [\xi_1 + 2\eta, \xi_2]$ we interchange the roles of $v_1$ and $v_2$). Note that $\xi_2$ is the inflection point of $v_2$, so $v_2$ is convex on $(0,\xi_2)$, which provides $v_2(x) \ge v_2(x_m) + v_2'(x_m)(x-x_m)$ for all $x \in (0,\xi_2)$. Therefore,
\[
v_2(\xi_1 + 3 \eta) \ge v_2(x_m) + v_2'(x_m)(\xi_1 + 3 \eta -x_m) \ge c(\xi_1 + 3 \eta -x_m) \ge c \eta > 0,
\]
for some positive $c$ in view of \eqref{vpbelow}. Now we reason as in Remark \ref{comparison3} (in particular we apply \eqref{comparison30} with $v_1$ and $v_2$ interchanged) to get
\[
v_1(\xi_1 + 3\eta) \le C_\infty e^{- \frac{c_2}{\sqrt{\nu}}} = o (\nu^{1/2}),
\]
but $C_g^{-1} v_1^2(\xi_2) \ge g_2(v_1^2(\xi_2)) = \lambda_2 \ge c_2 \nu$ (again by \eqref{Gass} and \eqref{lambdavsnu}), so $v_1(\xi_2) \ge v_1(\xi_1 + 3\eta)$ as $\nu \rightarrow 0$. Being $v_1$ decreasing, $\xi_2 \le \xi_1 + 3\eta = \xi_2 - \eta$, which is impossible. Then, $0 \le \xi_2 - \xi_1 \rightarrow 0$ follows, and the second assertion is proved as $x_m \rightarrow x_0$.
\end{proof}
\begin{proof}[Proof of Theorem \ref{thm:intro_S1}]
In view of Propositions \ref{prop:limiting_problem}, \ref{vprimebound} and Remark \ref{rem:H1conv},
the theorem will follow once we show that the following equalities hold:
\[
\frac{\ell_2}{\ell_1}= \left(\frac{\gamma_2}{\gamma_1} \right)^{2/3}
\qquad\text{and}\qquad
x_0 = \dfrac{\sqrt[3]{\gamma_2}}{\sqrt[3]{\gamma_1}+\sqrt[3]{\gamma_2}}.
\]
To this aim, we put together all the asymptotic information (as $\nu \rightarrow 0$) we obtained so far. Firstly, $\pi v_1^2(0) \sim 4 \sqrt{\ell_1}$ and $\pi v_2^2(1) \sim 4 \sqrt{\ell_2}$ by \eqref{V1} and \eqref{V2}. Hence, if we divide \eqref{Tvalue} by $\nu$ we obtain
\begin{equation}\label{ell1ell2}
\frac{\ell_1 \sqrt{\ell_1}}{\gamma_1} = \frac{\ell_2 \sqrt{\ell_2}}{\gamma_2},
\end{equation}
which is the first stated equality. Then, $x_0 = \frac{\pi}{2\sqrt{\ell_1}}$ by \eqref{xilimit} and \eqref{xilimit2}. Moreover,
\[
\frac{\pi}{2\sqrt{\ell_1}} + \frac{\pi}{2\sqrt{\ell_2}} = 1.
\]
By plugging \eqref{ell1ell2} in the last equality we conclude.
\end{proof}

\subsection*{Funding}
Work partially supported  by the PRIN-2012-74FYK7 Grant:
``Variational and perturbative aspects of nonlinear differential problems'',
by the ERC Advanced Grant  2013 n. 339958:
``Complex Patterns for Strongly Interacting Dynamical Systems - COMPAT'',
and by the INDAM-GNAMPA grant ``Analisi Globale, PDEs e Strutture Solitoniche'' (2015).

\small


\begin{thebibliography}{10}

\bibitem{ABC}
Y.~Achdou, M.~Bardi, and M.~Cirant.
\newblock Mean field games models of segregation.
\newblock preprint, 2015.

\bibitem{MR1225101}
A.~Ambrosetti and G.~Prodi.
\newblock {\em A primer of nonlinear analysis}, volume~34 of {\em Cambridge
  Studies in Advanced Mathematics}.
\newblock Cambridge University Press, Cambridge, 1993.

\bibitem{Bartsch:2015aa}
T.~Bartsch and L.~Jeanjean.
\newblock Normalized solutions for nonlinear {S}chr{\"o}dinger systems.
\newblock {\em arxiv: 1507.04649}, 07 2015.

\bibitem{Bartsch:2015ab}
T.~Bartsch, L.~Jeanjean, and N.~Soave.
\newblock Normalized solutions for a system of coupled cubic {S}chr{\"o}dinger
  equations on $\mathbb{R}^3$.
\newblock {\em arXiv:1506.02262}, 06 2015.

\bibitem{MR3021546}
H.~Berestycki, T.-C. Lin, J.~Wei, and C.~Zhao.
\newblock On phase-separation models: asymptotics and qualitative properties.
\newblock {\em Arch. Ration. Mech. Anal.}, 208(1):163--200, 2013.

\bibitem{MR3062741}
H.~Berestycki, S.~Terracini, K.~Wang, and J.~Wei.
\newblock On entire solutions of an elliptic system modeling phase separations.
\newblock {\em Adv. Math.}, 243:102--126, 2013.

\bibitem{MR2393430}
L.~A. Caffarelli and F.-H. Lin.
\newblock Singularly perturbed elliptic systems and multi-valued harmonic
  functions with free boundaries.
\newblock {\em J. Amer. Math. Soc.}, 21(3):847--862, 2008.

\bibitem{MR2278412}
L.~A. Caffarelli and J.-M. Roquejoffre.
\newblock Uniform {H}{\"o}lder estimates in a class of elliptic systems and
  applications to singular limits in models for diffusion flames.
\newblock {\em Arch. Ration. Mech. Anal.}, 183(3):457--487, 2007.

\bibitem{MR2090357}
S.-M. Chang, C.-S. Lin, T.-C. Lin, and W.-W. Lin.
\newblock Segregated nodal domains of two-dimensional multispecies
  {B}ose-{E}instein condensates.
\newblock {\em Phys. D}, 196(3-4):341--361, 2004.

\bibitem{MR3333058}
M.~Cirant.
\newblock Multi-population mean field games systems with {N}eumann boundary
  conditions.
\newblock {\em J. Math. Pures Appl. (9)}, 103(5):1294--1315, 2015.

\bibitem{MR1939088}
M.~Conti, S.~Terracini, and G.~Verzini.
\newblock Nehari's problem and competing species systems.
\newblock {\em Ann. Inst. H. Poincar{\'e} Anal. Non Lin{\'e}aire},
  19(6):871--888, 2002.

\bibitem{MR2146353}
M.~Conti, S.~Terracini, and G.~Verzini.
\newblock Asymptotic estimates for the spatial segregation of competitive
  systems.
\newblock {\em Adv. Math.}, 195(2):524--560, 2005.

\bibitem{MR2283921}
M.~Conti, S.~Terracini, and G.~Verzini.
\newblock Uniqueness and least energy property for solutions to strongly
  competing systems.
\newblock {\em Interfaces Free Bound.}, 8(4):437--446, 2006.

\bibitem{MR0288640}
M.~G. Crandall and P.~H. Rabinowitz.
\newblock Bifurcation from simple eigenvalues.
\newblock {\em J. Functional Analysis}, 8:321--340, 1971.

\bibitem{MR2831712}
E.~N. Dancer, K.~Wang, and Z.~Zhang.
\newblock Uniform {H}{\"o}lder estimate for singularly perturbed parabolic
  systems of {B}ose-{E}instein condensates and competing species.
\newblock {\em J. Differential Equations}, 251(10):2737--2769, 2011.

\bibitem{MR3127148}
E.~Feleqi.
\newblock The derivation of ergodic mean field game equations for several
  populations of players.
\newblock {\em Dyn. Games Appl.}, 3(4):523--536, 2013.

\bibitem{MR2251558}
A.~Henrot.
\newblock {\em Extremum problems for eigenvalues of elliptic operators}.
\newblock Frontiers in Mathematics. Birkh{\"a}user Verlag, Basel, 2006.

\bibitem{HCM2}
M.~Huang, P.~E. Caines, and R.~P. Malham{\'e}.
\newblock Large-population cost-coupled {LQG} problems with nonuniform agents:
  individual-mass behavior and decentralized {$\epsilon$}-{N}ash equilibria.
\newblock {\em IEEE Trans. Automat. Control}, 52(9):1560--1571, 2007.

\bibitem{HCM}
M.~Huang, R.~P. Malham{\'e}, and P.~E. Caines.
\newblock Large population stochastic dynamic games: closed-loop
  {M}c{K}ean-{V}lasov systems and the {N}ash certainty equivalence principle.
\newblock {\em Commun. Inf. Syst.}, 6(3):221--251, 2006.

\bibitem{LachapelleWolfram}
A.~Lachapelle and M.-T. Wolfram.
\newblock On a mean field game approach modeling congestion and aversion in
  pedestrian crowds.
\newblock {\em Transp. Res. Part B: Methodol.}, 45(10):1572 -- 1589, 2011.

\bibitem{jeux1}
J.-M. Lasry and P.-L. Lions.
\newblock Jeux \`a champ moyen. {I}. {L}e cas stationnaire.
\newblock {\em C. R. Math. Acad. Sci. Paris}, 343(9):619--625, 2006.

\bibitem{jeux2}
J.-M. Lasry and P.-L. Lions.
\newblock Jeux \`a champ moyen. {II}. {H}orizon fini et contr\^ole optimal.
\newblock {\em C. R. Math. Acad. Sci. Paris}, 343(10):679--684, 2006.

\bibitem{LasryLions}
J.-M. Lasry and P.-L. Lions.
\newblock Mean field games.
\newblock {\em Jpn. J. Math.}, 2(1):229--260, 2007.

\bibitem{LionsVideo}
P.-L. Lions.
\newblock Cours au coll\`ege de france.
\newblock http://www.college-de-france.fr.

\bibitem{MR2599456}
B.~Noris, H.~Tavares, S.~Terracini, and G.~Verzini.
\newblock Uniform {H}{\"o}lder bounds for nonlinear {S}chr{\"o}dinger systems
  with strong competition.
\newblock {\em Comm. Pure Appl. Math.}, 63(3):267--302, 2010.

\bibitem{MR2928850}
B.~Noris, H.~Tavares, S.~Terracini, and G.~Verzini.
\newblock Convergence of minimax structures and continuation of critical points
  for singularly perturbed systems.
\newblock {\em J. Eur. Math. Soc. (JEMS)}, 14(4):1245--1273, 2012.

\bibitem{MR3318740}
B.~Noris, H.~Tavares, and G.~Verzini.
\newblock Existence and orbital stability of the ground states with prescribed
  mass for the {$L^2$}-critical and supercritical {NLS} on bounded domains.
\newblock {\em Anal. PDE}, 7(8):1807--1838, 2014.

\bibitem{MR3393268}
B.~Noris, H.~Tavares, and G.~Verzini.
\newblock Stable solitary waves with prescribed {$L^2$}-mass for the cubic
  {S}chr{\"o}dinger system with trapping potentials.
\newblock {\em Discrete Contin. Dyn. Syst.}, 35(12):6085--6112, 2015.

\bibitem{MR3116007}
V.~Quitalo.
\newblock A free boundary problem arising from segregation of populations with
  high competition.
\newblock {\em Arch. Ration. Mech. Anal.}, 210(3):857--908, 2013.

\bibitem{MR0301587}
P.~H. Rabinowitz.
\newblock Some global results for nonlinear eigenvalue problems.
\newblock {\em J. Functional Analysis}, 7:487--513, 1971.

\bibitem{Soave:2015aa}
N.~Soave, H.~Tavares, S.~Terracini, and A.~Zilio.
\newblock H{\"o}lder bounds and regularity of emerging free boundaries for
  strongly competing {S}chr{\"o}dinger equations with nontrivial grouping.
\newblock {\em arxiv:1506.00800}, 06 2015.

\bibitem{MR3375537}
N.~Soave and A.~Zilio.
\newblock Uniform bounds for strongly competing systems: the optimal
  {L}ipschitz case.
\newblock {\em Arch. Ration. Mech. Anal.}, 218(2):647--697, 2015.

\bibitem{MR2384550}
J.~Wei and T.~Weth.
\newblock Asymptotic behaviour of solutions of planar elliptic systems with
  strong competition.
\newblock {\em Nonlinearity}, 21(2):305--317, 2008.

\end{thebibliography}

\medskip
\begin{flushright}
\noindent \verb"marco.cirant@unimi.it"\\
Dipartimento di Matematica, Universit\`a di Milano\\
via Cesare Saldini 50, 20133 Milano (Italy)

\smallskip

\noindent \verb"gianmaria.verzini@polimi.it"\\
Dipartimento di Matematica, Politecnico di Milano\\
piazza Leonardo da Vinci 32, 20133 Milano (Italy)
\end{flushright}

\end{document}